\documentclass{siamltex}
\usepackage{amsmath}       
\usepackage{amsfonts}
\usepackage{nccmath}     
\usepackage{amssymb}  
\usepackage{algorithm,algorithmicx,algpseudocode}
\usepackage{graphicx}
\usepackage{caption}
\usepackage{color}

\newcommand{\esssup}{\mathop{\mathrm{ess\,sup}}} 

\newcommand{\ltwo}[2]{\|{#1}\|_{#2}}

\newcommand{\ud}{\mathrm{d}}
\newcommand{\ndg}[1]{| \kern -.25mm \|{#1}| \kern -.25mm \|}

\DeclareMathOperator\diam{\mathrm{diam}}

\begin{document}
\pagestyle{myheadings}
\markboth{A.~CANGIANI, E.~H.~GEORGOULIS, I.~KYZA AND S.~METCALFE}
{ADAPTIVITY AND BLOW-UP DETECTION}

\title{Adaptivity and blow-up detection for \\ nonlinear evolution problems}
\author{Andrea Cangiani\thanks{Department of Mathematics, University of Leicester, University Road, Leicester LE1 7RH, United Kingdom} \and Emmanuil H.~Georgoulis\footnotemark[1] \and Irene Kyza\thanks{Department of Mathematics, University of Dundee, Nethergate, Dundee, DD1 4HN, United Kingdom}  \and Stephen Metcalfe\footnotemark[1]}

\maketitle

\begin{abstract}
{
This work is concerned with the development of a space-time adaptive numerical method, based on a rigorous a posteriori error bound, for a semilinear convection-diffusion problem which may exhibit blow-up in finite time. More specifically, a posteriori error bounds are derived in the $L^{\infty}(L^2)+L^2(H^1)$-type norm for a first order in time implicit-explicit (IMEX) interior penalty discontinuous Galerkin (dG) in space discretization of the problem, although the theory presented is directly applicable to the case of conforming finite element approximations in space.  The choice of the discretization in time is made based on a careful analysis of adaptive time stepping methods for ODEs that exhibit finite time blow-up. The new adaptive algorithm is shown to accurately estimate the blow-up time of a number of problems, including one which exhibits regional blow-up.
}
\end{abstract}
\begin{keywords}
finite time blow-up; conditional a posteriori error estimates; IMEX method; discontinuous Galerkin methods
\end{keywords}

\section{Introduction}

The numerical approximation of blow-up phenomena in partial differential equations (PDEs) is a challenging problem due to the high spatial and temporal resolution needed close to the blow-up time. Numerical methods giving good approximations close to the blow-up time include the rescaling algorithm of Berger and Kohn \cite{BK88,NZ14} and the MMPDE method \cite{BHR96,HMR08}. There is also work looking at the numerical approximation of blow-up in the  nonlinear Schr\"{o}dinger equation and its generalizations \cite{ADKM03,CS02,FI03,KMR11,Plex,TS92}. Other numerical methods for approximating blow-up can be found for a variety of different nonlinear PDEs \cite{ADR,ABN09,DPMF05,DKKV98,FGR02,G1,NB11} and ordinary differential equations (ODEs) \cite{H06,JW14, SF90}. Typically, these numerical methods rely on some form of theoretically justified rescaling but lack a rigorous justification as to whether the resulting numerical approximations are reasonable.  In contrast, our approach is to perform numerical rescaling of a simple numerical scheme in an adaptive space-time setting driven by rigorous a posteriori error bounds.

A posteriori error estimators for finite element discretizations of nonlinear para-bolic problems are available in the literature (e.g., \cite{B05, BM11,BBM05,CEGST14,DGN07,GM14,KNS04, V98f,V98e,V04}).  However, the literature on a posteriori error control for parabolic equations that exhibit finite time blow-up is very limited; to the best of our knowledge, only in \cite{K01} do the authors provide rigorous a posteriori error bounds for such problems. Using a semigroup approach, the authors of \cite{K01} arrive to \emph{conditional} a posteriori error estimates in the $L^\infty(L^\infty)$-norm for first and second order temporal semi-discretizations of a semilinear parabolic equation with polynomial nonlinearity. Conditional a posteriori error estimates have been derived in earlier works for several types of PDEs, see, e.g., \cite{CM,Fierro,KNS04,LaNo,MN2}; the estimates are called conditional because they only hold under a computationally verifiable smallness condition.

In this work, we derive a practical conditional a posteriori bound for a fully-discrete first order in time implicit-explicit (IMEX) interior penalty discontinuous Galerkin (dG) in space discretization of a non self-adjoint semilinear parabolic PDE with quadratic nonlinearity. The choice of an IMEX discretization and, in particular, the explicit treatment of the nonlinearity, offers advantages in the context of finite time blow-up {\bf --} this is highlighted below via the discretization of the related ODE problem with various time-stepping schemes. The choice of a dG method in space offers stability of the spatial operator in convection-dominated regimes on coarse meshes; we stress, however, that the theory presented below is directly applicable to the case of conforming finite element approximations in space. The conditional a posteriori error bounds are derived in the $L^{\infty}(L^2)+L^{2}(H^1)$-type norm. The derivation is based on energy techniques combined with the Gagliardo-Nirenberg inequality while retaining the key idea introduced in \cite{K01} {\bf --} judicious usage of Gronwall's lemma. A key novelty of our approach is the use of a \emph{local-in-time} continuation argument in conjunction with a space-time reconstruction. Global-in-time continuation arguments have been used to derive conditional a posteriori error estimates for finite element discretizations of PDEs with globally bounded solutions, cf. \cite{B05,GM14,KNS04}. A useful by-product of the local continuation argument used in this work is that it gives a natural stopping criterion for approach towards the blow-up time. The use of space-time reconstruction, introduced in \cite{LM06,MN03} for conforming finite element methods and in \cite{CGM14,GLV11} for dG methods, allows for the derivation of a posteriori bounds in norms weaker than $L^2(H^1)$ and offers great flexibility in treating general spatial operators and their respective discretizations.  

Furthermore, a space-time adaptive algorithm is proposed which uses the conditional a posteriori bound to control the time step lengths and the spatial mesh modifications. The adaptive algorithm is a non-trivial modification of typical adaptive error control procedures for parabolic problems. In the proposed adaptive algorithm, the tolerances are adapted in the run up to blow-up time to allow for larger absolute error in an effort to balance the relative error of the approximation. The space-time adaptive algorithm is tested on three numerical experiments, two of which exhibit point blow-up and one which exhibits regional blow-up. Each time the algorithm appears to detect and converge to the blow-up time without surpassing it. 

The remainder of this work is structured as follows. In Section \ref{ODE_sec}, we discuss the derivation of a posteriori bounds and algorithms for adaptivity for ODE problems whose solutions blow-up in finite time. Section \ref{prelim} sets out the model problem and introduces some necessary notation while Section \ref{dg_section} discusses the discretization of the problem. Within Section \ref{apost_bound_sec} the proof of the conditional a posteriori error bound is presented. An adaptive algorithm based on this a posteriori bound is proposed in Section \ref{adaptive_sec} followed by a series of numerical experiments in Section \ref{numer_exp_sec}. Finally, some conclusions are drawn in Section \ref{conclusions_sec}.

\section{Approximation of blow-up in ODEs}\label{ODE_sec}

Before proceeding with the a posteriori error analysis and adaptivity of the semilinear PDE, 
it is illuminating to consider the numerical approximation of blow-up in the context of ODEs. To this end, we first analyse the ODE initial value problem: find $u:[0,T]\to\mathbb{R}$ such that
\begin{equation}
\begin{aligned}
\label{ODE}
\frac{du}{dt} & = f(u) := \sum_{j=0}^{p} \alpha_j u^j, \quad \text{ in } (0,T], \\
u(0) & = u_0,
\end{aligned}
\end{equation}
with $p \geq 2$ a positive integer and coefficients $\alpha_i \geq 0$, $i=1,\dots,p-1$ and $\alpha_p > 0$ so that the solution blows up in finite time. Let $T^*$ denote the blow-up time of \eqref{ODE} and assume $T<T^*$. For $t<T^*$, $u(t)$ is a differentiable function \cite{Hu}. 

Let $0\le t^{k}\le T $, $0\le k\le N$ be defined by $t^{k} := t^{k-1} + \tau_{k}$ with $t^0 := 0$ and $t^N=T$ for some time step lengths $\tau_{k}>0$, $k=1,\dots,N$ with $\sum_{k=1}^N \tau_k =T$. We use the following three different one step schemes to approximate \eqref{ODE}. We set $U^0 := u_0$ and, for $k=1,\dots,N$,  we seek $U^{k}$ such that
\begin{equation}
\begin{aligned}
\label{ODEapprox}
\frac{U^{k} - U^{k-1}}{\tau_{k}} = F (U^{k-1},U^{k} ),
\end{aligned}
\end{equation}
with $F$ one of the following three classical approximations of $f$:
\begin{equation}
\label{differenceschemes}
\begin{aligned}
\mbox{Explicit Euler} \qquad F (U^{k-1},U^{k} )& =f (U^{k-1} ), \\ 
\mbox{Implicit Euler} \qquad F (U^{k-1},U^{k} )& =f (U^{k} ), \\ 
\mbox{Improved Euler} \qquad F (U^{k-1},U^{k} )& =1/2\left(f (U^{k-1} )+f \left(U^{k-1}+\tau_{k}f(U^{k-1})\right) \right).
\end{aligned}
\end{equation}

\subsection{An a posteriori error estimate}

We begin by defining $U:[0,T]\to \mathbb{R}$ by
 \begin{equation}
 \label{piecewiselinear}
 U(t):=\ell_{k-1}(t)U^{k-1}+\ell_{k}(t)U^{k},\qquad t\in (t^{k-1},t^{k} ],
 \end{equation}
where $\{\ell_{k-1}, \ell_{k}\}$ denotes the standard linear Lagrange interpolation basis defined on the interval $[t^{k-1},t^{k}]$, i.e., $U$ is the continuous piecewise linear interpolant through the points $(t^k,U^{k})$, $k=0,\dots,N$. Hence, \eqref{ODEapprox} can be equivalently written on each interval $(t^{k-1},t^{k}]$ as
 \begin{equation}
 \label{ODEapprox2}
\frac{dU}{dt}=F (U^{k-1},U^{k}).
 \end{equation}
Therefore, on each interval $(t^{k-1},t^{k}]$, the error $e:=u-U$ satisfies the equation
\begin{equation}
\label{ODEerror_one}
\frac{de}{dt} = f(u) - F (U^{k-1},U^{k})=f(U)+f'(U)e+\sum_{j=2}^p\frac{f^{(j)}(U)}{j!}e -  F (U^{k-1},U^{k}),
\end{equation}
with $f^{(j)}$ denoting the order $j$ derivative of $f$. Thus, upon defining the residual $\eta_{k} := f(U) - F (U^{k-1},U^{k})$, we obtain on each interval $(t^{k-1},t^{k}]$ the primary error equation:
\begin{equation}
\begin{aligned}
\label{ODEerror2}
\frac{de}{dt} = \eta_{k} + f'(U)e + \sum_{j=2}^p \frac{f^{(j)}(U)}{j!}e^j.
\end{aligned}
\end{equation}
Gronwall's inequality, therefore, implies that
\begin{equation}\label{ODEbound1}
|e(t)| \leq H_{k}(t)G_{k}\phi_{k},
\end{equation}
where
\begin{equation}
\label{boundnotation}
\begin{aligned}
H_{k}(t) & := \exp \bigg (\sum_{j=2}^p \int_{t^{k-1}}^{t} \! \frac{|f^{(j)}(U)|}{j!}|e|^{j-1} \, \ud s \bigg ), \\
G_{k} & := \exp \bigg(\int_{t^{k-1}}^{t^{k}} \! |f'(U)| \, \ud s \bigg), \\
\phi_{k} & :=  |e(t^{k-1})| + \int_{t^{k-1}}^{t^{k}} \! |\eta_{k}| \, \ud s.
\end{aligned}
\end{equation}
From~\eqref{ODEbound1}, we derive an a posteriori bound by a \emph{local continuation argument}. To this end, we define the set
\begin{equation}\label{subinterval}
\begin{aligned}
I_{k} := \Big\{t \in [t^{k-1},t^{k}] :  \max_{s \in [t^{k-1},t]} |e(s)| \leq \delta_{k}G_{k}\phi_{k}\Big\},
\end{aligned}
\end{equation}
for some $\delta_{k}>1$ to be chosen below; note that $I_k\subset[t^{k-1},t^{k}]$ and that $I_k$ is closed since the error function $e$ is continuous. The main idea is to use the continuity of the error function $e$ and to choose $\delta_{k}$ implying that  $I_{k}=[t^{k-1},t^{k}]$ for each $k=1,\dots, N$. Further, we will choose $\delta_{k}$ to be a computable bound of $H_{k}$ thereby arriving to an a posteriori bound.

\begin{theorem}[Conditional a posteriori estimate]\label{apostest1}
Let $u$ be the exact solution of \eqref{ODE}, $\{U^k\}_{k=0}^N$ the approximations produced by \eqref{ODEapprox} and $U$ the piecewise linear interpolant \eqref{piecewiselinear}. Then, for $k=1,\dots,N$, the following a posteriori estimate holds:
\begin{equation}
\begin{aligned}
\label{aposterrorest}
\max_{t\in[t^{k-1},t^{k}]}|e(t)|  \leq \delta_{k}G_{k}\phi_{k},
\end{aligned}
\end{equation}
provided that $\delta_{k}>1$ is chosen so that
\begin{equation}
\label{deltaequation}
\sum_{j=2}^p (\delta_{k}G_{k}\phi_{k})^{j-1}\int_{t^{k-1}}^{t^{k}} \! \frac{|f^{(j)}(U(s))|}{j!} \, \ud s - \log(\delta_{k}) = 0.
\end{equation}
\end{theorem}
\begin{proof}
For $k=1,\dots,N$, let $I_{k}$ be as in \eqref{subinterval} where $\delta_{k}>1$ is chosen to satisfy 
\begin{equation}
\label{deltacondition2}
\exp\bigg(\sum_{j=2}^p(\delta_{k}G_{k}\phi_{k})^{j-1}\int_{t^{k-1}}^{t^{k}}\frac{|f^{(j)}(U(s))|}{j!}
\,\ud s\bigg)
\le(1-\alpha)\delta_{k},
\end{equation}
for some $0<\alpha<1$.
The interval $I_{k}$ is closed and non-empty since $t^{k-1}\in I_{k}$; hence, it attains a maximum $t_*^{k}:=\max I_{k}$. Suppose that $t^{k}_*<t^k$. In view of the definition of $H_{k}$, we have
\begin{equation}
\label{Hbound}
H_{k}(t_*^{k}) \leq \exp\bigg( \sum_{j=2}^p (\delta_{k}G_{k}\phi_{k})^{j-1}\int_{t^{k-1}}^{t^{k}} \! \frac{|f^{(j)}(U(s))|}{j!} \, \ud s \bigg)\le (1-\alpha)\delta_{k}<\delta_k,
\end{equation}
as $t_*^{k}\in I_{k}$.  Application of~\eqref{Hbound} to \eqref{ODEbound1} yields
\begin{equation}
\label{estcond1}
\max_{t\in[t^{k-1},t^{k}_*]}|e(t)|<\delta_{k}G_{k}\phi_{k}.
\end{equation}
This implies that $t^{k}_*$ cannot be the maximal element of $I_k$ {\bf --} a contradiction. Hence, $t^{k}_*=t^{k}$ and, thus, $I_{k}= \big[t^{k-1},t^{k} \big]$. Considering the case with equality in \eqref{deltacondition2} and taking $\alpha\to 0$, we arrive at \eqref{deltaequation} and the proof is complete.
\end{proof}

Choosing $\delta_k>1$ satisfying \eqref{deltaequation} is equivalent to finding a root (ideally the smallest one) of
$$\sum_{j=2}^p\bigg((G_k\phi_k)^{j-1}\int_{t^{k-1}}^{t^k}\frac{f^{(j)}\left(U(s)\right)}{j!}\, \ud s\bigg)x^{j-1}-\log(x),$$
in the interval $(1,+\infty)$.
This is only possible if the coefficients of $x^j$, $j=1,\ldots,p-1$, are ``sufficiently small'', i.e., only provided that the time steps length $\tau_k$ is small enough.  In this sense, \eqref{aposterrorest} is a conditional a posteriori bound. In practice, condition \eqref{deltaequation} is implemented by applying Newton's method. If for some $k$ the time step length $\tau_k$ is not small enough, Newton's method does not converge and the procedure terminates (cf. Algorithms~1 and~2 below). With the aid of the next lemma, we state a precise condition on the time step lengths $\tau_k$ which indeed ensures that \eqref{deltaequation} has a root $\delta_k>1.$
\begin{lemma}
If $\sum_{j=1}^pj C_j\mathrm{e}^j \le 1$ then $s(x) = \sum_{j=1}^p C_j x^j-\log(x)$ with $C_j> 0$, $j=1,\ldots,p$, $p\in\mathbb{N}$ has a root in $(1,+\infty).$
\end{lemma}
\begin{proof}
We begin by noting that $s(x)$ is continuous in $[1,+\infty)$ and differentiable in $(1,+\infty)$ with $s(1)=\sum_{j=1}^pC_j>0$ and $\lim_{x\to+\infty} s(x)=+\infty.$ Therefore $s(x)$ has a root in $(1,+\infty)$ if and only if it attains a nonpositive minimum in this interval. Differentiating $s(x)$ gives $s'(x)=\sum_{j=1}^p j C_j x^{j-1}-x^{-1}.$ Since the coefficients  $C_j$ satisfy $\sum_{j=1}^pj C_j\mathrm{e}^j \le 1$, we observe that
$$s'(\mathrm{e})=\sum_{j=1}^pjC_j\mathrm{e}^{j-1}-\mathrm{e}^{-1}\le 0.$$
Also, $\lim_{x\to +\infty}s'(x)>0.$ Hence, there exists a critical point $x_*\in[\mathrm{e},+\infty)$ satisfying 
\begin{equation}
\label{criticalpoint}
\sum_{j=1}^p jC_jx_*^j=1.
\end{equation}
All that remains is to prove that $s(x_*)\le 0$. Indeed, \eqref{criticalpoint} leads to 
$$\sum_{j=1}^pC_jx_*^j\le \sum_{j=1}^pj C_jx_*^j=1\le\log(x_*),$$
where the last inequality holds because $x_*\ge \mathrm{e}$. From the above relation, we readily conclude that $s(x_*)\le 0$ and the proof is complete.
\end{proof}

The above lemma gives a sufficient condition on when \eqref{deltaequation} can be satisfied. In particular, condition \eqref{deltaequation} can always be made to be satisfied provided that the time step length $\tau_k$ is chosen such that 
$$\sum_{j=2}^p\frac{j-1}{j!}\left(G_k\phi_k \mathrm{e}\right)^{j-1}\int_{t^{k-1}}^{t^k}{|f^{(j)}\left(U(s)\right)|} \, \ud s \le1.$$ 

Returning back to Theorem~\ref{apostest1}, we note that this gives a recursive procedure for the estimation of the error on each subinterval $[t^{k-1},t^{k}]$. Indeed, the term $|e(t^{k-1} )|$ in $\phi_{k}$ is estimated using the error estimator from the previous time step with $e(0) = 0$.

\subsection{Adaptivity}
 Based upon the a posteriori  error estimator presented in Theorem~\ref{apostest1}, we propose Algorithm 1 for advancing towards the blow-up time.

\begin{algorithm} \label{ODEalgorithm1}
  \begin{algorithmic}[1]
     \State {\bf Input:} $f$, $F$, $u_0$, $\tau_1$, ${\tt tol}$.
     \State  Compute $U^1$ from $U^0$.
     \While {$\displaystyle \int_{t^0}^{t^{1}} \! |\eta_{1}| \, \ud s > {\tt tol}$}
     \State $\tau_1 \leftarrow \tau_1/2$.
     \State Compute $U^1$ from $U^0$.
\EndWhile
\State Compute $\delta_{1}$.
\State Set $k = 0$.
     \While {$\delta_{k+1}$ exists}
     \State $k \leftarrow k+1$.
     \State $\tau_{k+1} = \tau_k$.
     \State Compute $U^{k+1}$ from $U^k$.
     \While {$\displaystyle \int_{t^k}^{t^{k+1}} \! |\eta_{k+1}| \, \ud s > {\tt tol}$}
     \State $\tau_{k+1} \leftarrow \tau_{k+1}/2$.
     \State Compute $U^{k+1}$ from $U^k$.
\EndWhile
\State Compute $\delta_{k+1}$.
\EndWhile
\State {\bf Output:} $k$, $t^{k}$.
  \end{algorithmic}
  \caption{ODE Algorithm 1}
\end{algorithm}

Assuming that the adaptive algorithm outputs successfully at time $T=T({\tt tol},N)$ for a given tolerance ${\tt tol}$ and after total number of time steps $N$ then we are interested in observing the order of convergence as $T\to T^*$ with respect to $N$. To this end, we define the function
$ \lambda({\tt tol},N) := |T^* - T({\tt tol},N)|$ where $T^*$ is the blow-up time of \eqref{ODE} and we numerically investigate the rate $r>0$ such that 
\begin{equation}
\lambda({\tt tol},N) \propto N^{-r}.
\end{equation}
One may initially expect that $r$ would be equal to the order of the time-stepping scheme used. To gain insight into the rate convergence of $\lambda$, we apply Algorithm 1 to \eqref{ODE} with $f(u)=u^p$ for $p=2,3$ and $u(0)=1$ for each of the three time-stepping schemes \eqref{differenceschemes}. The computed rates of convergence $r$ under Algorithm 1 are given in Table \ref{data1}.
\begin{table}[ht]
\caption{Algorithm 1 Results} % title of Table
\centering % used for centering table
\begin{tabular}{c c c c} % centered columns (4 columns)
\hline\hline %inserts double horizontal lines
Method &$ p = 2$ & $p = 3$ \\ % inserts table
%heading
\hline % inserts single horizontal line
Implicit Euler & $r \approx 0.66$ & $r \approx 0.79$ \\ % inserting body of the table
Explicit Euler & $r \approx 1.35$ & $r \approx 1.60$ \\
Improved Euler & $r \approx 1.2$ & $r \approx 1.48$ \\
\hline %inserts single line
\end{tabular}
\label{data1}
\end{table}

Somewhat surprisingly at first sight, the explicit Euler scheme performs significantly better than the implicit Euler scheme. This fact can be explained by looking back at the derivation of the error estimator. The explicit Euler scheme always underestimates the true solution $u$ \cite{SF90}.

This, in turn, implies that $\delta_{k+1}$ is correcting for the fact that $G_{k+1}$ is underestimating the true blow-up rate resulting in a tight a posteriori error bound and, thus, explaining the high rate of convergence of $\lambda$. When using the implicit Euler method, on the other hand, $G_{k+1}$ overestimates the true blow-up rate \cite{SF90} thereby conferring no additional benefit. 

Note also that for both the implicit and improved Euler methods, the rate of convergence $r$  is less than their formal orders of convergence, i.e., first and second order, respectively. Moreover, one would expect a faster approach to the blow-up time using the second order improved Euler compared to the first order explicit Euler scheme. This unexpected behaviour is due to the way the tolerance is utilized in Algorithm 1. Indeed, Algorithm 1 aims to reduce the error under an absolute tolerance {\tt tol}; this is the standard practice in adaptive algorithms applied to linear problems. In the context of blow-up problems, however, the presence of the growth factor $G_{k+1}$ cannot be neglected; requiring the adaptivity to be driven by an absolute tolerance in the run up to the blow up time results in excessive over-refinement and, thus, loss of the expected rate of convergence. To address this issue, we propose Algorithm 2 which increases ${\tt tol}$ proportionally to $G_{k+1}$ allowing for control of the relative error (cf. line 19 in Algorithm 2).

\begin{algorithm} \label{ODEalgorithm2}
  \begin{algorithmic}[1]
     \State {\bf Input:} $f$, $F$, $u_0$, $\tau_1$, ${\tt tol}$.
     \State Compute $U^1$ from $U^0$.
     \While {$\displaystyle \int_{t^0}^{t^{1}} \! |\eta_{1}| \, \ud s > {\tt tol}$}
     \State $\tau_1 \leftarrow \tau_1/2$.
     \State Compute $U^1$ from $U^0$.
\EndWhile
\State Compute $\delta_{1}$.
\State ${\tt tol} = G_1*{\tt tol}.$
\State Set $k=0$.
     \While {$\delta_{k+1}$ exists}
     \State $k \leftarrow k+1$.
     \State $\tau_{k+1} = \tau_k$.
     \State Compute $U^{k+1}$ from $U^k$.
     \While {$\displaystyle \int_{t^k}^{t^{k+1}} \! |\eta_{k+1}| \, \ud s > {\tt tol}$}
     \State $\tau_{k+1} \leftarrow \tau_{k+1}/2$.
     \State Compute $U^{k+1}$  from $U^k$.
\EndWhile
\State Compute $\delta_{k+1}$.
\State ${\tt tol} = G_{k+1}*{\tt tol}.$
\EndWhile
\State {\bf Output:} $k$, $t^{k}$.
  \end{algorithmic}
  \caption{ODE Algorithm 2}
\end{algorithm}

\begin{table}[ht]
\caption{Algorithm 2 Results} % title of Table
\centering % used for centering table
\begin{tabular}{c c c c} % centered columns (4 columns)
\hline\hline %inserts double horizontal lines
Method & $p = 2$ & $p = 3$ \\
\hline % inserts single horizontal line
Implicit Euler & $r \approx 1.00$ & $r \approx 1.00$ \\ % inserting body of the table
Explicit Euler & $r \approx 1.45$ & $r \approx 1.43$ \\
Improved Euler & $r \approx 2.03$ & $r \approx 2.03$ \\
\hline %inserts single line
\end{tabular}
\label{data2}
\end{table}
The rates of convergence $r$ of $\lambda$ under Algorithm 2 are given in Table \ref{data2}. The theoretically conjectured orders of convergence for both the implicit and improved Euler schemes are recovered while the explicit Euler method still outperforms its expected rate. In the case $p=3$ (cubic nonlinearity) and for the explicit Euler method only,  Algorithm 1 converges somewhat faster than Algorithm 2. The reason for this behaviour is unclear and requires further investigation.

\section{Model problem}\label{prelim}

Let $\Omega \subset \mathbb{R}^2$ be the computational domain which is assumed to be a bounded  polygon with Lipschitz boundary $\partial\Omega$. We denote the standard $L^2$-inner product on $\omega\subseteq\Omega$ by $(\cdot,\cdot)_{\omega}$ and the standard $L^2$-norm by $\ltwo{\cdot}{\omega}$; when $\omega = \Omega$ these will be abbreviated to $(\cdot,\cdot)$ and $\ltwo{\cdot}{}$, respectively. We shall also make use of the standard Sobolev spaces $W^{k,p}(\omega)$ along with the standard notation $L^p(\omega)=W^{0,p}(\omega)$, $1\le p\le\infty$; $H^k(\omega):=W^{k,2}(\omega)$, $k\ge 0$; and $H^1_0(\Omega)$ denoting the subspace of $H^1(\Omega)$ consisting of functions vanishing on the boundary $\partial\Omega$. For $T>0$ and a real Banach space $X$ with norm $\|\cdot\|_X$, we define the spaces $L^p(0,T;X)$ consisting of all measurable functions $v: [0,T]\to X$ for which
\begin{equation}
\begin{aligned}
\notag
\|v\|_{L^p(0,T;X)}&:=\bigg(\int_0^T \! \|v(t)\|_{X}^p \, dt\bigg)^{1/p}<\infty, \qquad && \text{for } 1\le p< +\infty, \\
\|v\|_{L^{p}(0,T;X)}&:=\esssup_{0\le t\le T}\|v(t)\|_{X}<\infty,& & \text{for }  p = +\infty.
\end{aligned}
\end{equation} 
We also define {$H^1(0,T,X):=\{u\in L^2(0,T;X): u_t\in  L^2(0,T;X)\}$} and we denote by $C(0,T;X)$ the space of continuous functions $v:[0,T] \rightarrow X$ such that
\begin{equation}
\begin{gathered}
\notag
||v||_{C(0,T;X)} :=\max_{0 \leq t \leq T}{||v(t)||_X} < \infty\mbox{.}
\end{gathered}
\end{equation}
The model problem consists of finding $u:\Omega\times(0,T] \to\mathbb{R}$ such that
\begin{equation}\label{model_strong}
\begin{aligned}
\frac{\partial{u}}{\partial{t}} - \varepsilon\Delta{u}+ {\bf a} \cdot \nabla{u}+f(u) &= 0 \qquad && \text{in }  \Omega\times(0,T]  \mbox{,} \\ u &=0 \mbox{ } && \text{on }  \partial\Omega\times (0,T] \mbox{,} \\ u(\cdot,0) &=u_0 \mbox{ } && \text{in } \bar{\Omega}\mbox{,}
\end{aligned}
\end{equation}
for $f(u) = f_0 - u^2$ and where $u_0 \in H^1_0(\Omega)$, 
$\varepsilon>0$, ${\bf a} \in [C(0,T;W^{1, \infty}(\Omega))]^2$ and $f_0 \in C(0,T;L^2(\Omega))$. For simplicity of the presentation only, we shall also assume that $\nabla \cdot {\bf a} = 0$ although this is not an essential restriction to the analysis that follows.

The weak form of \eqref{model_strong} reads: find $u\in L^2(0,T;H^1_0(\Omega))\cap H^1(0,T;L^2(\Omega))$ such that for almost every $t \in (0,T]$ we have
\begin{equation}\label{blowup_model_weak}
\bigg( \frac{\partial{u}}{\partial{t}},v \bigg)+B(t;u,v)+(f(t;u),v) =0 \qquad \forall v \in H^1_0(\Omega),
\end{equation}
where
\begin{equation}
\notag
B(t;u,v) := \int_{\Omega} \! (\varepsilon \nabla{u} - {\bf a}u) \cdot \nabla{v} \, \ud x.
\end{equation}
Under the above assumptions and for any $t \in (0,T]$ the bilinear form $B$ is coercive in and $H^1_0(\Omega)$, viz.,  $B(t;v,v) \geq \varepsilon\ltwo{\nabla v}{}^2$, for all $v \in H^1_0(\Omega)$.

\section{Discretization}\label{dg_section}
Consider a shape-regular mesh $\zeta=\{K\}$ of $\Omega$ with $K$ denoting a generic element that is constructed via affine mappings $F_{K}:\hat{K}\to K$ with non-singular Jacobian where $\hat{K}$ is the reference triangle or the reference square. The mesh is allowed to contain a uniformly fixed number of regular hanging nodes per edge. On $\zeta$, we define the finite element space
\begin{equation}\label{eq:FEspace}
\mathbb{V}_h(\zeta) := \{v \in L^2(\Omega):v|_K\circ F_K \in \mathcal{P}^p(\hat{K}), \, K \in \zeta \},
\end{equation}
 with $\mathcal{P}^p(K)$ denoting the space of polynomials of total degree  $p$ or of degree $p$ in each variable if $\hat{K}$ is the reference triangle or the reference square, respectively. 
In what follows, we shall often make use of the orthogonal $L^2$-projection onto the finite element space $\mathbb{V}_h^{k}$, which we will denote by $\Pi^{k}$.

The set of all edges in the triangulation $\zeta$ is denoted by $\mathcal{E}(\zeta)$ while $\mathcal{E}^{int}(\zeta)\subset\mathcal{E}(\zeta)$ stands for the subset of all interior edges. Given $K \in \zeta$ and $E \in \mathcal{E}(\zeta)$, we set $h_K:=\diam(K)$ and $h_E:=\diam(E)$, respectively; we also denote the outward unit normal to the boundary $\partial{K}$ by ${\bf n}_K$. Given an edge $E\in \mathcal{E}^{int}(\zeta)$ shared by two elements $K$ and $K'$, a vector field ${\bf v}\in [H^{1/2}(\Omega)]^2$ and a scalar field $v\in H^{1/2}(\Omega)$, we define jumps and averages of ${\bf v}$ and $v$ across $E$ by
\begin{equation}
\begin{aligned}
\notag
 \{{\bf v}\} & := \frac{1}{2}({\bf v}|_{E\cap\bar{K}}+{\bf v}|_{E\cap\bar{K}'} ), \qquad &
   [{\bf v}]  :=& {\bf v}|_{E\cap\bar{K}} \cdot {\bf n}_K+ {\bf v}|_{E\cap\bar{K}'} \cdot {\bf n}_{K'},  \\
  \{v\} & := \frac{1}{2}( v|_{E\cap\bar{K}}+ v|_{E\cap\bar{K}'} ), \qquad &
   [v]   := &  v|_{E\cap\bar{K}}  {\bf n}_K+  v|_{E\cap\bar{K}'}  {\bf n}_{K'}.
 \end{aligned}
\end{equation}
If $E\subset \partial\Omega$, we set $\{{\bf v}\}:={\bf v}$, $[{\bf v}]:={\bf v} \cdot {\bf n}$, $\{v\}:= v$ and $[ v]:= v {\bf n}$ with ${\bf n}$ denoting the outward unit normal to the boundary $\partial\Omega$. The inflow and outflow parts of the boundary $\partial\Omega$ at time $t$, respectively, are defined by
\begin{equation}
\notag
 \partial\Omega^t_{in} := \{x \in \partial\Omega : {\bf a}(x,t) \cdot {\bf n}(x) < 0 \}\mbox{,} \quad \partial\Omega^t_{out} := \{x \in \partial\Omega : {\bf a}(x,t) \cdot {\bf n}(x) \geq 0 \}\mbox{.} 
\end{equation}
Similarly, the inflow and outflow parts of an element $K$ at time $t$ are defined by
\begin{equation}
\notag
\partial K^t_{in} := \{x \in \partial K : {\bf a}(x,t) \cdot {\bf n}_K(x) < 0 \}\mbox{,} \quad \partial K^t_{out}  := \{x \in \partial K : {\bf a}(x,t) \cdot {\bf n}_K(x) \geq 0 \}\mbox{.}
\end{equation}

We consider an implicit-explicit (IMEX) space-time discretization of \eqref{blowup_model_weak} consisting of implicit treatment for the linear convection-diffusion terms and explicit treatment for the nonlinear reaction term which was shown to be beneficial in Section \ref{ODE_sec}. For the spatial discretization, we use a standard (upwinded) interior penalty discontinuous Galerkin method, detailed below, to ensure stability of the spatial operator in convection-dominated regimes.

To this end, we consider a subdivision of $[0,T]$ into time intervals of lengths $\tau_1,\dots,\tau_N$ such that $\sum_{j=1}^N{\tau_j}=T$ for some integer $N \geq 1$ and we set $t^0 := 0$ and $t^k := \sum_{j=1}^{k}\tau_{j}$, $k=1,\dots,N$. Let $\zeta^0$ denote an initial spatial mesh of $\Omega$ associated with the time $t^0=0$. To each time $t^k$, $k=1,\dots,N$, we associate the spatial mesh $\zeta^k$ of $\Omega$ which is assumed to have been obtained from $\zeta^{k-1}$ by local refinement and/or coarsening. Each mesh $\zeta^{k}$ is assigned the finite element space $\mathbb{V}_h^k := \mathbb{V}_h(\zeta^k)$ given by~\eqref{eq:FEspace}. For brevity, let $ {\bf a}^k:={\bf a}(\cdot,t^k )$ and $f^k:=f(\cdot,t^k;U_h^k)$. Finally, for $t \in (t^{k-1},t^{k}]$, $\Gamma(t)$ will denote the union of all edges in the coarsest common refinement $\zeta^{k-1} \cup \zeta^{k}$ of $\zeta^{k-1}$ and $\zeta^{k}$.

The IMEX dG method then reads as follows. Let $U_h^0$ be a projection of $u_0$ onto $\mathbb{V}_h^0$. For $k=1,\dots, n$, find $U_h^{k} \in \mathbb{V}_h^{k}$ such that
\begin{equation}\label{dg_fd}
\bigg(\frac{U_h^{k}-U_h^{k-1}}{\tau_{k}},v_h^{k}\bigg )+B (t^{k};U_h^{k},v_h^{k} )+K_h (U_h^{k},v_h^{k} ) + (f^{k-1},v_h^{k} )=0,
\end{equation}
for all $v_h^k \in \mathbb{V}_h^k$ where
\begin{equation}
\begin{aligned}
\notag
B (t^k;U_h^k,v_h^k )  :=&  \sum_{K \in \zeta^k} \int_{K} \! (\varepsilon \nabla U_h^k  -{\bf{a}}U_h^k ) \cdot \nabla v_h^k \, \ud x+\sum_{E \in \mathcal{E}(\zeta^k)}\frac{ \gamma \varepsilon}{h_E}\int_{E} \!  [U_h^k ] \cdot  [v_h^k ] \, \ud s \\ &+ \sum_{K \in \zeta^k}\int_{\partial{K}^{t^k}_{out}} \! U_h^k[{\bf a} v_h^k] \, \ud s,  \\ 
K_h (U_h^k,v_h^k )  := & -\sum_{E \in \mathcal{E}(\zeta^k)}\int_{E} \! \{\varepsilon \nabla U_h^k \} \cdot [v_h^k ] + \{\varepsilon \nabla v_h^k \} \cdot [U_h^k ] \, \ud s.
\end{aligned}
 \end{equation}
We shall choose $U_h^0$ as the orthogonal $L^2$-projection of $u_0$ onto $\mathbb{V}_h^0$, that is $U_h^0 :=\Pi^{0}u_0$, although other projections onto $\mathbb{V}_h^0$ can also be used. In standard fashion, the penalty parameter, $\gamma$, is chosen large enough so that the operator $B+K_h$ is coercive on $\mathbb{V}_h^k$ (see, e.g., \cite{DE12}).

\section{An a posteriori bound}\label{apost_bound_sec}

In the context of the elliptic reconstruction framework \cite{LM06,MN03}, we require an a posteriori error bound for a related stationary problem. To that end, we consider a generalization of the error bound introduced in \cite{SZ09}; the proof of such bound is completely analogous and is, therefore, omitted for brevity.

\begin{theorem}\label{elliptic_apost}
Given $t \in (0,T]$ and $g \in L^2(\Omega)$, let $u^s \in H^1_0(\Omega)$ be the exact solution of the elliptic problem
\vspace{-0.2cm}
$$
B(t;u^s,v)=(g,v),\vspace{-0.2cm}
$$
for all $v\in H^1_0(\Omega)$
and let $u^s_h\in \mathbb{V}_h$ such that
\vspace{-0.2cm}
$$
B(t;u^s_h,v_h)+K_h(u^s_h,v_h)=(g,v_h) ,\vspace{-0.2cm}
$$
for all  $v_h\in \mathbb{V}_h$ be its dG approximation.
Then the following a posteriori error bound holds for any $0\ne v \in H_0^1(\Omega)$:
\begin{equation}\label{apost_stationary}
\begin{aligned}
\notag
\bigg(\frac{B(u^s-u_h^s,v)}{\sqrt{\varepsilon}\|\nabla v\|}\bigg)^{2}  \lesssim & \sum_{K\in \zeta}\frac{h^2_K}{\varepsilon}\|g+ \varepsilon \Delta u^s_h-{\bf a} \cdot \nabla u^s_h \|_{K}^2 +\sum_{E\in \mathcal{E}^{int}(\zeta)} \varepsilon h_E\| [\nabla u^s_h] \|_{E}^2 \\ & +\sum_{E\in \mathcal{E}(\zeta)}  \frac{\gamma\varepsilon}{h_E}\| [ u^s_h] \|_{E}^2+\frac{h_E}{\varepsilon}\| [{\bf a} u^s_h] \|_{E}^2.
\end{aligned}
\end{equation}
\end{theorem}
The symbols $\lesssim$ and $\gtrsim$ used above and throughout the rest of this section denote inequalities true up to a constant independent of the data $\varepsilon$, $\bf a$, $f$, the exact and numerical solutions $u$, $u_h$, and the local mesh-sizes and time step lengths.
\begin{definition}
We denote by $A^k \in \mathbb{V}_h^k$ the unique solution of the problem
\begin{equation}
\notag
B (t^{k};U_h^k,v_h^k )+K_h (U_h^k,v_h^k )= (A^k,v_h^k ) \qquad
\forall v_h^k \in \mathbb{V}_h^k.
\end{equation}
\end{definition}
For $k \geq 1$, we observe that $A^{k} = -\Pi^{k}f^{k-1} - (U_h^{k}-\Pi^{k}U_h^{k-1})/\tau_{k}$ from \eqref{dg_fd}. 

\begin{definition}\label{ellipticreconstr2} We define the elliptic reconstruction $w^{k} \in H^1_0(\Omega)$, $k=0,\dots, N$, to be the unique solution of the elliptic problem
$$B(t^{k};w^{k},v )=(A^k,v ) \qquad \forall v \in H^1_0(\Omega).$$
\end{definition}
Crucially, the dG discretization of the elliptic problem in Definition \ref{ellipticreconstr2} is equal to $U_h^k$ and so $B(t^k;w^k-U_h^k,v)$ can be estimated by Theorem \ref{elliptic_apost}.

At each time step $k$, we decompose the dG solution $U_h^k$ into a conforming part $U_{h,c}^k \in H^1_0(\Omega) \cap \mathbb{V}_h^k$ and a non-conforming part $U_{h,d}^k \in \mathbb{V}_h^k$ such that $U_h^k = U_{h,c}^k + U_{h,d}^k$. Further, for $t \in (t^{k-1},t^{k}]$, we define $U_h(t)$ to be the linear interpolant with respect to $t$ of the values $U_h^{k-1}$ and $U_h^{k}$, viz., 
\vspace{-0.2cm}
$$U_h(t):=\ell_{k-1}(t)U_h^{k-1}+\ell_{k}(t)U_h^{k},\vspace{-0.2cm}$$
where, as before, $\{\ell_{k-1}, \ell_{k}\}$ denotes the standard linear Lagrange interpolation basis defined on the interval $[t^{k-1},t^{k}]$. We define $U_{h,c}(t)$ and $U_{h,d}(t)$  analogously. We then decompose the error $e := u-U_h=e_c-U_{h,d}$ with $e_c:=u-U_{h,c}$ and we denote the elliptic error by $\theta^{k} := w^{k}-U_h^{k}$.

\begin{theorem} \label{ncbounds} Given $t \in [t^{k-1},t^{k}]$, there exists a decomposition of $U_h$, as described above, such that the following bounds hold for each element $K \in \zeta^{k-1} \cup \zeta^k$:
\begin{equation}
\begin{aligned}
\notag
||\nabla U_{h,d}||^2_{K} & \lesssim \sum_{E \subset \tilde{K}_E} h_E^{-1}||[U_h]||^2_{E}, \\
||U_{h,d}||^2_{K}  & \lesssim \sum_{E \subset \tilde{K}_E} h_E||[U_h]||^2_{E}, \\
||U_{h,d}||_{L^{\infty}(K)} & \lesssim ||[U_h]||_{L^{\infty}(\tilde{K}_E)},
\end{aligned}
\end{equation}
where $\tilde{K}_E := \big\{ \bigcup E : \bar{K} \cap \bar{E} \neq \emptyset, \mbox{ } E \in \mathcal{E}(\zeta^k \cup \zeta^{k+1}) \big\}$ denotes the edge patch of the element $K$ {\bf --} the union of all edges with a vertex on $\partial K$.
\end{theorem}
\begin{proof}
See \cite{KP03} for the first two estimates and \cite{DG01} for the final estimate.
\end{proof}

\medskip

\begin{lemma}\label{blowup_error_eq_simple_fd}
Let $t \in (t^{k-1},t^{k} ]$ then for any $v \in H^1_0(\Omega)$ we have
\begin{equation*}\label{error_relation_fd}
\bigg(\frac{\partial{e}}{\partial{t}},v \bigg)+B(t;e,v)+(f(t;u)-f(t;U_h),v)=\bigg(-f(t;U_h)-\frac{\partial{U_h}}{\partial{t}},v \bigg)-B(t;U_h,v).
\end{equation*}
\end{lemma}
\begin{proof}
This follows from \eqref{blowup_model_weak}. 
\end{proof}

\medskip

From Lemma \ref{blowup_error_eq_simple_fd} we obtain
\begin{equation}
\begin{aligned}
&\bigg(\frac{\partial{e}}{\partial{t}},v \bigg)+B(t;e,v )+(f(t;u )-f(t;U_h),v)=-\bigg(A^{k}+f^{k-1}+\frac{\partial{U_h}}{\partial{t}},v \bigg)
\\
&+B(t^{k};\theta^{k},v )-B(t;U_h,v )+B(t^{k};U_h^{k},v )+(f^{k-1}-f(t;U_h),v ),
\end{aligned}
\end{equation}
which, upon straightforward manipulation, gives
\begin{equation}
\begin{aligned}
\label{equation2}
&\bigg(\frac{\partial{e}}{\partial{t}},v \bigg)+B(t;e,v)+(f(t;u )-f (t;U_h ),v)=-\bigg(A^{k}+f^{k-1}+\frac{\partial{U_h}}{\partial{t}},v \bigg)\\&+\ell_{k-1}B(t^{k-1};\theta^{k-1},v)+\ell_{k}B(t^{k};\theta^{k},v )-B(t;U_h,v )+\ell_{k-1}B(t^{k-1};U_h^{k-1},v)\\&+\ell_{k}B(t^{k};U_h^{k},v)+(f^{k-1}-f (t;U_h )+\ell_{k-1}(A^{k}-A^{k-1}),v).
\end{aligned}
\end{equation}
In what follows, it will be convenient to define the a posteriori error estimator through three constituent terms $\eta_I$, $\eta_A$, and $\eta_B$.
The first part of the estimator is the initial condition estimator, $\eta_I$, given by
\begin{equation}
\notag
\eta_I  := \bigg ( \|e(0)\|^2 +\sum_{E \in \mathcal{E}(\zeta^{0})} h_E \|[U_h^0]\|^2_{E} \bigg )^{1/2}.
\end{equation}
Both remaining parts, $\eta_A$ and $\eta_B$, are the sum of a number of terms related to either a space or time discretisation error, identified by a subscript $S$ or $T$, respectively.
In this way, for $t \in (t^{k-1},t^{k}]$, $\eta_A$ is given by
\begin{equation}
\begin{aligned}
\notag
\eta_A & :=  \ell_{k-1} \eta_{S_1,k-1}+ \ell_k \eta_{S_1,k} + \eta_{S_2,k} + \eta_{T_1,k},
\end{aligned}
\end{equation}
where
\begin{equation}
\begin{aligned}
\notag
\eta_{S_1,k}   := &  \bigg(\sum_{K \in \zeta^{k}} \frac{h_K^2}{\varepsilon}\|A^{k} + \varepsilon \Delta U_h^{k} - {\bf a}^{k} \cdot \nabla U_h^{k} \|^2_{K}+\sum_{E \in \mathcal{E}^{int}(\zeta^{k})} \varepsilon h_E \|[\nabla U_h^{k}]\|^2_{E} \\ & +\sum_{E \in \mathcal{E}(\zeta^{k})}  \frac{\gamma \varepsilon}{h_E} \|[U_h^{k}]\|^2_{E}+ \frac{h_E}{\varepsilon}\|[{\bf a}^{k} U_h^{k}]\|^2_{E}  \bigg)^{1/2}, \\
\eta_{S_2,k}  & :=  \bigg(\sum_{K \in \zeta^{k-1} \cup \zeta^{k}} \frac{h_K^2}{\varepsilon}\|f^{k-1}-\Pi^{k}f^{k-1}- (U_h^{k-1} - \Pi^{k}U_h^{k-1})/\tau_{k} \|^2_{K} \bigg)^{1/2}, \\ 
\eta_{T_1,k} & := {\varepsilon}^{-1/2}\|\ell_{k-1} \big({\bf a}^{k-1}-{\bf a} \big)U_h^{k-1}+\ell_{k}\big({\bf a}^{k}-{\bf a} \big)U_h^{k} \|,
\end{aligned}
\end{equation}
while $\eta_B$ is given by
\begin{equation}
\notag
\eta_B  := \eta_{S_3,k}+\eta_{S_4,k}+\eta_{T_2,k},
\end{equation}
where
\begin{equation}
\begin{aligned}
\notag
\eta_{S_3,k}  & := \bigg (\sum_{K \in \zeta^{k-1} \cup \zeta^{k}} \sum_{E \subset \tilde{K}_E} \sigma^2_K h_E \|[U_h]\|^2_{E} \bigg )^{1/2}, \\
\eta_{S_4,k}  & :=  \bigg ( \sum_{E \subset \Gamma(t)} h_E \| [(U_h^{k}-U_h^{k-1})/\tau_{k} ]\|^2_{E} \bigg)^{1/2}, \\
\eta_{T_2,k} & := \|f^{k-1}-f(t;U_h)+\ell_{k-1} (A^{k}-A^{k-1})\|,
\end{aligned}
\end{equation}
with
\begin{equation}
\begin{aligned}
\notag
\sigma_K := 2||U_h||_{L^{\infty}(K)}+||[U_h]||_{L^{\infty}(\tilde{K}_E)}.
\end{aligned}
\end{equation}
With the above notation at hand, we go back to \eqref{equation2} and bound the first term on the right-hand side using the definition of $A^k$, the orthogonality property of the $L^2$-projection and the Cauchy-Schwarz inequality:
\begin{equation}
\begin{aligned}
\bigg(A^{k}+f^{k-1}+\frac{\partial{U_h}}{\partial{t}},v \bigg)  &= \bigg(f^{k-1}-\Pi^kf^{k-1}-\frac{U_h^{k-1}-\Pi^kU_h^{k-1}}{\tau_k},v-\Pi^{k}v \bigg)\\
&\lesssim \eta_{S_2,k}\sqrt{\varepsilon}||\nabla v||.
\end{aligned}
\end{equation}
The next two terms give rise to parts of the space estimator via Theorem \ref{elliptic_apost}:
\begin{equation}
\ell_{k-1}B(t^{k-1};\theta^{k-1},v)+\ell_k B(t^k;\theta^k,v ) \lesssim  (\ell_{k-1} \eta_{S_1,k-1}+ \ell_k \eta_{S_1,k})\sqrt{\varepsilon}||\nabla v||.
\end{equation}
Using the definition of the bilinear form $B$ and the Cauchy-Schwarz inequality, the final four terms give rise to the time estimator:
\begin{equation}
\begin{aligned}
\ell_{k-1}B(t^{k-1};U_h^{k-1},v )+\ell_kB(t^k;U_h^k,v) -B(t;U_h,v) & \leq \eta_{T_1,k}\sqrt{\varepsilon}||\nabla v||, \\
(f^{k-1}-f(t;U_h)+\ell_{k-1}\big(A^{k}-A^{k-1} \big),v) & \leq \eta_{T_2,k}||v||.
\end{aligned}
\end{equation}
Setting $v=e_c$ in \eqref{equation2}, using the results above along with the coercivity of the bilinear form $B$ and the Cauchy-Schwarz inequality, we obtain
\begin{equation}
\begin{aligned}
& \frac{1}{2}\frac{d}{dt}\|e_c\|^2+\varepsilon\|\nabla e_c\|^2+(f(t;u )-f(t;U_h ),e_c )\lesssim   \bigg (\|\frac{\partial U_{h,d}}{\partial t}\|+\eta_{T_2,k} \bigg)\|e_c\| \\ & +  ( \ell_{k-1}\eta_{S_1,k-1} + \ell_k\eta_{S_1,k} + \eta_{S_2,k}+  \eta_{T_1,k} )\sqrt{\varepsilon}\|\nabla e_c\| +B(t;U_{h,d},e_c).
\end{aligned}
\end{equation}
Application of Theorem \ref{ncbounds} implies that
\begin{equation}
\begin{aligned}
& \frac{1}{2}\frac{d}{dt}\|e_c\|^2+\varepsilon\|\nabla e_c\|^2+(f(t;u )-f(t;U_h ),e_c )\lesssim  (\eta_{S_4,k}+\eta_{T_2,k})||e_c|| \\ &+  ( \ell_{k-1}\eta_{S_1,k-1} + \ell_k\eta_{S_1,k} + \eta_{S_2,k}+  \eta_{T_1,k} )\sqrt{\varepsilon}\|\nabla e_c\|.
\end{aligned}
\end{equation}
Thus, we conclude that
\begin{equation}
\begin{aligned}
\label{equation3}
&\frac{1}{2}\frac{d}{dt}\|e_c\|^2+\frac{\varepsilon}{2}\|\nabla e_c\|^2+(f(t;u)-f(t;U_h),e_c ) \lesssim  \frac{1}{2}\eta^2_A + \eta_B\|e_c\|.
\end{aligned}
\end{equation}
We must now  deal with the nonlinear term on the left-hand side of~\eqref{equation3}. We begin by noting that
\begin{equation}
\label{horriblenonlinear}
 (f(t;u)-f(t;U_h),e_c) = (f(t;e_c-U_{h,d}+U_h)-f(t;U_h),e_c) = T_1 + T_2 ,
\end{equation}
where
\begin{equation}
\begin{aligned}
\notag
 T_1 & :=  (2U_h U_{h,d},e_c ) - (U_{h,d}^2,e_c ) , \\
T_2 & :=  -(2U_he_c,e_c ) + (2e_c U_{h,d},e_c) -  (e_c^2,e_c ).
\end{aligned}
\end{equation}
Upon writing the contributions to $T_1$ elementwise and using Theorem \ref{ncbounds}, we have
\begin{equation}
\begin{aligned}
\label{Tbound1}
|T_1| & \leq \bigg (\sum_{K \in \zeta^{k-1} \cup \zeta^{k}}\!\!  \big(2\|U_h\|_{L^{\infty}(K)}+\|U_{h,d}\|_{L^{\infty}(K)}  \big)^2\|U_{h,d}\|^2_{K} \bigg )^{1/2}\|e_c\| \\
& \lesssim \eta_{S_3,k}\|e_c\|.
\end{aligned}
\end{equation}
To bound $T_2$, we use H{\"o}lder's inequality along with Theorem \ref{ncbounds} to conclude that
\begin{equation}
\label{Tbound2}
|T_2|  \lesssim \big (2\|U_h\|_{L^{\infty}(\Omega)} + \|[U_h]\|_{L^{\infty}(\Gamma(t))} \big)\|e_c\|^2+\|e_c\|^3_{L^3(\Omega)}.
\end{equation}
Combining \eqref{equation3}, \eqref{horriblenonlinear}, \eqref{Tbound1} and \eqref{Tbound2} we obtain 
\begin{equation}
\label{PDEerrorequation}
\frac{d}{dt}\|e_c\|^2+\varepsilon\|\nabla e_c\|^2  \leq  C\eta^2_A + 2C\eta_B \|e_c\| + 2\sigma_{\Omega}\|e_c\|^2 + 2\|e_c\|^3_{L^3(\Omega)},
\end{equation}
with $\sigma_{\Omega} := 2\|U_h\|_{L^{\infty}(\Omega)} +C\|[U_h]\|_{L^{\infty}(\Gamma(t))}$ where $C>0$ is a constant that is independent of $\varepsilon$, $\bf a$, $f$, $u$, $U_h$ and the local mesh-sizes and time step lengths. For $v\in H^1_0(\Omega)$, the Gagliardo-Nirenberg inequality $ \|v\|^3_{L^3(\Omega)}  \leq C_{\rm GN}\|v\|^2\|\nabla v\| $  implies that
\begin{equation}
\label{gagnir}
 \|e_c\|^3_{L^3(\Omega)}  \leq C_{\rm GN}\|e_c\|^2\|\nabla e_c\| \leq \frac{\varepsilon}{2}\|\nabla e_c\|^2 + \frac{C_{\rm GN}^2}{2\varepsilon}\|e_c\|^4.
\end{equation}
Thus,
\begin{equation}
\label{equation4}
\frac{d}{dt}\|e_c\|^2  \leq  C\eta^2_A + 2C\eta_B \|e_c\| + 2\sigma_{\Omega}\|e_c\|^2 +C_{\rm GN}^2 \varepsilon^{-1}\|e_c\|^4.
\end{equation}
To deal with the $L^2$-norms of $e_c$ appearing on the right-hand side, we use a variant of Gronwall's inequality. 
\begin{theorem}
\label{Gronwall}
Let $T>0$ and suppose that $c_0$ is a constant, $c_1,c_2\in L^1(0,T)$ are non-negative functions and that $u\in W^{1,1}(0,T)$ is a non-negative function satisfying
\begin{equation}
\notag
u^2(T) \leq c^2_0 + \int_0^T \! c_1(s)u(s) \, \ud s + \int_0^T \! c_2(s)u^2(s) \, \ud s,
\end{equation}
then
\begin{equation}
\notag
u(T) \leq \bigg (|c_0|+\frac{1}{2}\int_0^T \! c_1(s) \, \ud s  \bigg ) \exp \bigg (\frac{1}{2}\int_0^T \! c_2(s) \, \ud s \bigg ).
\end{equation}
\end{theorem}
\begin{proof}
See Theorem 21 in \cite{D01}.
\end{proof}

Application of Theorem \ref{Gronwall} to \eqref{equation4} for $t \in (t^{k-1},t^{k}]$ yields
\begin{equation}
\label{equation5}
\|e_c(t)\|  \leq \mathcal{H}_{k}(t)\mathcal{G}_{k}\Phi_{k},
\end{equation}
where
\begin{equation}
\begin{aligned}
\notag
\Phi_{k} &:= \bigg (\|e_c (t^{k-1} )\|^2 + C\int_{t^{k-1}}^{t^{k}} \! \eta_A^2 \, \ud s \bigg)^{1/2}+ C \int_{t^{k-1}}^{t^{k}} \! \eta_B \, \ud s, \\
\mathcal{G}_{k} & := \exp \bigg(\int_{t^{k-1}}^{t^{k}} \! \sigma_{\Omega} \, \ud s \bigg ), \\
\mathcal{H}_{k}(t) & := \exp \bigg ({C_{GN}^2}{\varepsilon}^{-1}\int_{t^{k-1}}^t \! \|e_c\|^2 \, \ud s  \bigg).
\end{aligned}
\end{equation}
To remove the non-computable term $\mathcal{H}_{k}$ from \eqref{equation5}, we use a continuation argument. We define the set
\begin{equation}
\notag
\mathcal{I}_{k}  :=  \{t \in [t^{k-1},t^{k} ]: \|e_c\|_{L^{\infty}(t^{k-1},t;L^2(\Omega))} \leq \delta_{k}\mathcal{G}_{k} \Phi_{k}  \},
\end{equation}
where, analagous to the ODE case, $\delta_{k} > 1$ should be chosen as small as possible. $\mathcal{I}_{k}$ is non-empty (since $t^{k-1} \in \mathcal{I}_{k}$) and bounded and, thus, attains some maximum value. Let $t^*=\max \mathcal{I}_k$ and assume that $t^* < t^{k}$. Then, from \eqref{equation5}, we have
\begin{equation}
\begin{aligned}
\label{equation6}
\|e_c\|_{L^{\infty}(t^{k-1},t^*;L^2(\Omega))} & \leq \mathcal{H}(t^*)\mathcal{G}_{k}\Phi_{k} \\ 
& \leq \exp\big (C_{\rm GN}^2 {\varepsilon}^{-1}\tau_{k}\|e_c\|_{L^{\infty}(t^{k-1},t^*;L^2(\Omega))}^2  \big)\mathcal{G}_{k}\Phi_{k} \\ 
& \leq \exp\big (C_{\rm GN}^2  {\varepsilon}^{-1}\tau_{k}\delta^2_{k}\mathcal{G}^2_{k}\Phi^2_{k} \big )\mathcal{G}_{k}\Phi_{k}.
\end{aligned}
\end{equation}
Now, suppose $\delta_k>1$ is chosen such that
\begin{equation}
\exp\big (C_{\rm GN}^2  {\varepsilon}^{-1}\tau_{k}\delta^2_{k}\mathcal{G}^2_{k}\Phi^2_{k} \big ) \le (1-\alpha) \delta_{k},
\end{equation}
for some $0<\alpha<1$ then \eqref{equation6} gives 
\begin{equation}
\|e_c\|_{L^{\infty}(t^{k-1},t^*;L^2(\Omega))}\le(1-\alpha) \delta_{k} \mathcal{G}_{k}\Phi_{k}<\delta_{k} \mathcal{G}_{k}\Phi_{k},
\end{equation}
which, in turn, implies that $t^*$ cannot be the maximal value of $t$ in $\mathcal{I}_{k}$ {\bf --} a contradiction. Hence $\mathcal{I}_{k} = [t^{k-1},t^{k}]$ and we have the desired error bound once $\delta_{k}$ is selected. Taking $\alpha\to 0$, we can select $\delta_{k}>1$ to be the smallest root of 
\begin{equation}
\label{cond2}
C_{\rm GN}^2  {\varepsilon}^{-1}\tau_{k}\delta^2_{k}\mathcal{G}^2_{k}\Phi^2_{k} - \log(\delta_{k}) = 0.
\end{equation}
Finally, we estimate $\Phi_1$. Application of Theorem \ref{ncbounds} and the triangle inequality yields

\begin{equation}
\|e_c(0)\|^2  \lesssim \|e(0)\|^2 + \|U_{h,d}(0)\|^2 \leq C\eta_I^2.
\end{equation}
Therefore, if we redefine $\Phi_1$ to be 
\begin{equation}
\notag
\Phi_{1} := \bigg (C\eta^2_I + C\int_{t^0}^{t^1} \! \eta_A^2 \, \ud s \bigg)^{1/2}+ C\int_{t^0}^{t^{1}} \! \eta_B \, \ud s,
\end{equation}
we have
\begin{equation}
\|e_c(t^1)\|  \leq \|e_c\|_{L^{\infty}(t^0,t^1;L^2(\Omega))} \leq \Psi_1,
\end{equation}
where $\Psi_1 := \delta_1\mathcal{G}_1\Phi_1$. In the same way, if we redefine 
\begin{equation}
\begin{aligned}
\notag
\Phi_{k} & := \bigg (\Psi^2_{k-1} + C\int_{t^{k-1}}^{t^{k}} \! \eta_A^2 \, \ud s \bigg)^{1/2}+ C \int_{t^{k-1}}^{t^{k}} \! \eta_B \, \ud s, \\
\Psi_{k} &:= \delta_{k}\mathcal{G}_{k}\Phi_{k},
\end{aligned}
\end{equation}
we have
\begin{equation}
\label{final_eq}
\|e_c(t^{k})\| \leq \|e_c\|_{L^{\infty}(t^{k-1},t^{k};L^2(\Omega))} \leq \Psi_{k}.
\end{equation}
Hence, we have shown the following result.
\begin{theorem}
\label{MainTheorem}
The error of the IMEX dG discretization of problem \eqref{blowup_model_weak}, given by~\eqref{dg_fd}, satisfies
\begin{equation}
\notag
\|e\|_{L^{\infty}(0,T;L^2(\Omega))} \lesssim \Psi_N + \esssup_{0 \leq t \leq T} 
\bigg( \sum_{E \subset \Gamma(t)} h_E \|[U_h]\|^2_{E} \bigg)^{1/2},
\end{equation}
providing that the solution to \eqref{cond2} exists for all time steps.
\end{theorem}
\begin{proof}
Follows from \eqref{final_eq}, the triangle inequality, and Theorem \ref{ncbounds}.
\end{proof}

\smallskip

The estimator produced above is suboptimal with respect to the mesh-size as it is only spatially optimal in the $L^2(H^1)$-norm. It is possible to conduct a continuation argument for the $L^2(H^1)$-norm rather than the $L^{\infty}(L^2)$-norm if one desires a spatially optimal error estimator; this is stated for completeness in the theorem below. However, the resulting $\delta$ equation was observed to be more restrictive with regards to how quickly the blow-up time is approached. For this reason, we opt to use the a posteriori error estimator of Theorem~\ref{MainTheorem} in the adaptive algorithm introduced in the next section.

\smallskip

\begin{theorem}
The error of the IMEX dG discretization of problem \eqref{blowup_model_weak}, given by~\eqref{dg_fd}, satisfies
\begin{equation}
\notag
\bigg( \|e(T)\|^2 + \int_0^T \varepsilon||\nabla e_c||^2 \, \ud t \bigg)^{1/2} \lesssim \sum_{k=1}^N \Psi_k + \esssup_{0 \leq t \leq T} 
\bigg( \sum_{E \subset \Gamma(t)} h_E \|[U_h]\|^2_{E} \bigg)^{1/2}.
\end{equation}
Furthermore, close to the blow-up time where $\|e(T)\| = \|e\|_{L^{\infty}(0,T;L^2(\Omega))}$ we have
\begin{equation}
\notag
\bigg(  \|e\|^2_{L^{\infty}(0,T;L^2(\Omega))} + \int_0^T \varepsilon||\nabla e_c||^2 \, \ud t \bigg)^{1/2} \lesssim \sum_{k=1}^N \Psi_k + \esssup_{0 \leq t \leq T} 
\bigg( \sum_{E \subset \Gamma(t)} h_E \|[U_h]\|^2_{E} \bigg)^{1/2},
\end{equation}
where $\Psi_k$, $k=1,\dots,N$, is defined recursively with $\Psi_0 = C\eta_I$ and
\begin{equation}
\begin{aligned}
\notag
\Phi_{k} & := \bigg (\Psi^2_{k-1} + C\int_{t^{k-1}}^{t^{k}} \! \eta_A^2 \, \ud s + C \int_{t^{k-1}}^{t^{k}} \! \eta^2_B \, \ud s \bigg)^{1/2}, \\
{G}_{k} & := \exp(\tau_k / 2) \exp \bigg(\int_{t^{k-1}}^{t^{k}} \! \sigma_{\Omega} \, \ud s \bigg ), \\
\Psi_{k} &:= \delta_{k}\mathcal{G}_{k}\Phi_{k},
\end{aligned}
\end{equation}
provided that $\delta_k > 1$ which is the smallest root of the equation
\begin{equation}
\begin{aligned}
\notag
C_{GN}\varepsilon^{-1/2}\tau^{1/2}_{k}\delta_{k}G_{k}\Phi_{k} - \log(\delta_{k})=0,
\end{aligned}
\end{equation}
exists for all time steps.
\end{theorem}
\begin{proof}
The proof is completely analogous to that of Theorem \ref{MainTheorem} and follows from \eqref{PDEerrorequation} by conducting a continuation argument for the $L^{\infty}(L^2)+L^2(H^1)$-norm.
\end{proof}

\section{An adaptive algorithm}\label{adaptive_sec}
We shall now proceed by stating our space-time adaptive algorithm for problems with finite time blow-up.

The algorithm is based on Algorithm 2 from Section 2 and the space-time adaptive algorithm for linear evolution problems presented in \cite{CGM14}. It makes use of different terms in the a posteriori bound in Theorem \ref{MainTheorem} to take automatic decisions on space-time refinement and coarsening. The pseudocode describing the adaptive algorithm is given in Algorithm 3.

\begin{algorithm} \label{adaptive_algorithm}
  \begin{algorithmic}[1]
     \State {\bf Input:} $\varepsilon$, ${\bf a}$, $f_0$, $u_0$, $\Omega$,  $\tau_1$, $\zeta^0$, $\gamma$, ${\tt ttol^+}$, ${\tt ttol^-}$, ${\tt stol^+}$, ${\tt stol^-}$.
     \State Compute $U_h^0$.
 \State Compute $U_h^1$ from $U_h^0$.
    \While {$\displaystyle \int_{t^0}^{t^1} \! \eta^2_{T_2,1} \, ds > {\tt ttol^+} \text{ OR } \max_K \eta^2_{S_1,1} |_K > {\tt stol^+} $}
    \State Modify $\zeta^0$ by refining all elements such that $\eta^2_{S_1,1} |_K > {\tt stol^+}$ and coarsening all elements such that $\eta^2_{S_1,1} |_K < {\tt stol^-}$.
\If   {$\displaystyle \int_{t^0}^{t^1} \! \eta^2_{T_2,1} \, ds > {\tt ttol^+}$}
   
 \State $\tau_{1} \leftarrow \tau_{1}/2$.

    \EndIf

     \State Compute $U_h^0$.
 \State Compute $U_h^1$ from $U_h^0$.
    \EndWhile

\State Compute $\delta_1$.

 \State Multiply ${\tt ttol^+}$, ${\tt ttol^-}$, ${\tt stol^+}$, ${\tt stol^-}$ by the factor $G_{1}$.

     \State  Set $j = 0$,  $\zeta^1 = \zeta^0$.

\While {$\delta_{j+1}$ exists}
\State $ j \leftarrow j+1$. 
\State $\tau_{j+1} = \tau_j$. 

    \State Compute $U_h^{j+1}$ from $U_h^j$.
    \If {  $\displaystyle\int_{t^j}^{t^{j+1}} \! \eta^2_{T_2,j+1} \, ds> {\tt ttol^+}$ }
     
 \State $\tau_{j+1} \leftarrow \tau_{j+1}/2$.
\State Compute $U_h^{j+1}$ from $U_h^j$.
    \EndIf
    \If {  $\displaystyle\int_{t^j}^{t^{j+1}} \! \eta^2_{T_2,j+1} \, ds < {\tt ttol^-}$ }
     
 \State $\tau_{j+1} \leftarrow 2\tau_{j+1}$.
\State Compute $U_h^{j+1}$ from $U_h^j$.
    \EndIf
    \State Form $\zeta^{j+1}$ from $\zeta^j$ by refining all elements such that $\eta^2_{S_1,j+1} |_K > {\tt stol^+}$ and coarsening all elements such that $\eta^2_{S_1,j+1} |_K < {\tt stol^-}$.
    \State Compute $U_h^{j+1}$ from $U_h^j$.
    \State Compute $\delta_{j+1}$.
 \State Multiply ${\tt ttol^+},{\tt ttol^-},{\tt stol^+},{\tt stol^-}$ by the factor $G_{j+1}$.
\EndWhile

\State {\bf Output:} $j$, $t^j$, $||U_h (t^j )||_{L^{\infty}(\Omega)}$.
  \end{algorithmic}
  \caption{Space-time adaptivity}
\end{algorithm}

The term $\eta_{S_1,k}$ drives both local mesh refinement and coarsening. The elements are refined, coarsened or left unchanged depending on two spatial thresholds ${\tt stol^+}$ and ${\tt stol^-}$. Similarly, the term $\eta_{T_2,k}$ is used to drive temporal refinement and coarsening subject to two temporal thresholds ${\tt ttol^+}$ and ${\tt ttol^-}$.

\section{Numerical Experiments}\label{numer_exp_sec}

We shall investigate numerically the a posteriori bound presented in Theorem \ref{MainTheorem} and the performance of the adaptive algorithm through an implementation based on the {\tt deal.II} finite element library \cite{BHK07}. All the numerical experiments have been performed using the  high performance computing facility ALICE at the University of Leicester. The following settings are common to all the numerical experiments presented. We use polynomials of degree five, hence $p=5$ throughout. This particular choice provides a good compromise between run time and spatial discretization error for the problems considered. Furthermore, it permits us to analyse the asymptotic temporal behaviour of the adaptive algorithm. 

Correspondingly, we set $\gamma = 30$ to ensure coercivity of the discrete bilinear form. We also set ${\tt ttol}^- = 0.01*{\tt ttol}^+$ and ${\tt stol}^- = 10^{-6}*{\tt stol}^+$ as, respectively, the temporal and spatial coarsening parameters. The initial mesh $\zeta^0$ is chosen to be a $4 \times 4$ uniform quadrilateral mesh and the initial time step length $\tau_1$ is chosen so that the first computed numerical approximation is before the expected blow-up time. The unknown constants in the a posteriori bound are set equal to one as is the constant $C_{GN}$ in  \eqref{cond2}; the above conventions  are deemed reasonable for the practical implementation of the a posteriori bound.

\subsection{Example 1}\label{num_ex_one}
We begin by considering a standard reaction-diffusion semilinear PDE problem whose blow-up behaviour is theoretically well understood. This is given by setting $\Omega = (-4,4)^2$, $\varepsilon = 1$, ${\bf a} = (0,0)^T$, $f_0 = 0$ and $u_0 = 10 \mathrm{e}^{-2(x^2+y^2)}$. The initial condition $u_0$ is chosen to be a Gaussian blob centred on the origin that is chosen large enough so that the solution exhibits blow-up; the blow-up set consists of a single point corresponding to the centre of the Gaussian.

To assess the asymptotic behaviour of the error estimator, we fix a very small spatial threshold so as to render the spatial contribution to both the error and the estimator negligible. We then vary the temporal threshold and record how far the algorithm is able to advance towards the blow-up time. The results are given in Table \ref{blowupdata1}.

\begin{table}[ht]
\caption{Example 1 Results} % title of Table
\centering % used for centering table
\begin{tabular}{c c c c c} % centered columns (4 columns)
\hline\hline %inserts double horizontal lines
${\tt ttol^+}$ & Time Steps & Estimator & Final Time & $||U_h(T)||_{L^{\infty}(\Omega)}$ \\ % inserts table
%heading
\hline % inserts single horizontal line
1 & 3 & 9.5 & 0.09375 & 12.244 \\ % inserting body of the table
0.125 & 8 & 24.6 & 0.12500 & 14.742 \\
$(0.125)^2$ & 19 & 54.0 & 0.14844 & 18.556 \\
$(0.125)^3$ & 42 & 66.7 & 0.16406 & 23.468 \\
$(0.125)^4$ & 92 & 218.5 & 0.17969 & 32.108 \\
$(0.125)^5$ & 195 & 1142.4 & 0.19043 & 44.217 \\
$(0.125)^6$ & 405 & 1506.0 & 0.19775 & 60.493 \\
$(0.125)^7$ & 832 & 1754.1 & 0.20313 & 83.315 \\
$(0.125)^8$ & 1698 & 5554.2 & 0.20728 & 117.780 \\
$(0.125)^9$ & 3443 & 6020.4 & 0.21014 & 165.833 \\
$(0.125)^{10}$ & 6956 & 33426.7 & 0.21228 & 238.705 \\
$(0.125)^{11}$ & 14008 & 36375.0 & 0.21375 & 343.078 \\
$(0.125)^{12}$ & 28151 & 66012.8 & 0.21478 & 496.885 \\
$(0.125)^{13}$ & 56489 & 157300.0 & 0.21549 & 722.884 \\
\hline %inserts single line
\end{tabular}
\label{blowupdata1}
\end{table}

For the present case (problems without convection), it is known that the solution to \eqref{blowup_model_weak} has the same asymptotic behaviour as the solution to \eqref{ODE} with respect to the time variable~\cite{Hu}. Thus, we would expect an effective estimator to yield similar rates for $\lambda$, the difference between the true and numerical blow-up time, to those seen in Section \ref{ODE_sec}. Although the true blow-up time for this problem is not known, we observe from Table \ref{blowupdata1} that 
\begin{equation}
\notag
\|U_h\|_{L^{\infty}(0,T;L^{\infty}(\Omega))} \propto N^{1/2}.
\end{equation}
From \cite{Hu}, we know the relationship between the magnitude of the exact solution in the $L^{\infty}(L^{\infty})$-norm and the distance from the blow-up time. Thus, under the assumption that the numerical solution is scaling like the exact solution, we have
\begin{equation}
\notag
\lambda({\tt ttol^+},N) \approx \|u\|^{-1}_{L^{\infty}(0,T;L^{\infty}(\Omega))} \approx \|U_h\|^{-1}_{L^{\infty}(0,T;L^{\infty}(\Omega))}.
\end{equation}
Therefore, we conjecture that
\begin{equation}
\begin{aligned}
\notag
\lambda({\tt ttol^+},N) \propto N^{-1/2}.
\end{aligned}
\end{equation}
Note that the conjectured convergence rate is slower than the comparable results in Section \ref{ODE_sec}; a possible explanation for this will be given in the concluding remarks.

Next we investigate the \emph{numerical blow-up rate} of $||U_h(t)||_{L^{\infty}(\Omega)}$. In particular, we are interested in checking if the numerical blow-up rate coincides with the theoretical one. For this particular example, it is well known, cf. \cite{MZ98,MerleZaag}, that close to the blow-up time $||u(t)||_{L^{\infty}(\Omega)}$ behaves as
$$\|u(t)\|_{L^\infty(\Omega)}\sim\frac1{T^*-t},$$
where $T^*$ denotes the blow-up time. Let us denote the \emph{numerical blow-up time} by $t^*$ which we compute as follows. For the last numerical experiment (with ${\tt ttol^+}= (0.125)^{13}$), we assume that there exists a constant $C_N$ such that
$$\|U_h(t)\|_{L^\infty(\Omega)}=C_N\frac1{t^*-t},\quad t=t^{N-1}, T.$$
Then $t^*$ is computed by
$$t^*=\frac{T\|U_h(T)\|_{L^\infty(\Omega)}-t^{N-1}\|U_h(t^{N-1})\|_{L^\infty(\Omega)}}{\|U_h(T)\|_{L^\infty(\Omega)}-\|U_h(t^{N-1})\|_{L^\infty(\Omega)}}.$$ 
For this example, the above relation gives $t^*=0.217055.$ 

\begin{figure}[h!]
\centering
\includegraphics[scale=0.45]{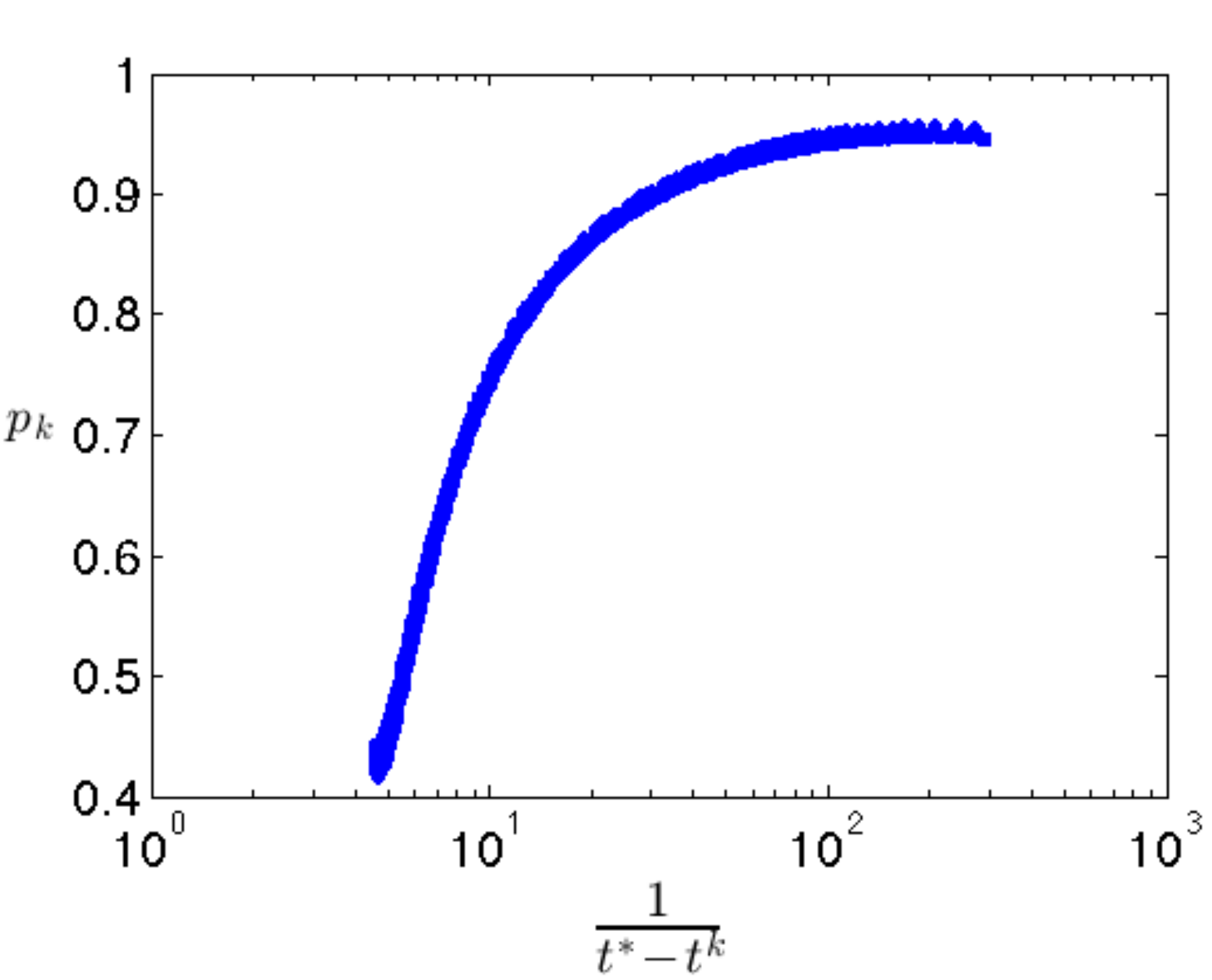}
\caption{Example 1: Numerical blow-up rate.}
\label{numericalblowuprate}
\end{figure}

To compute the numerical blow-up rate, we consider all the time nodes $t^k,\, k=0,\ldots,N$ with $t^N=0.21375$ (corresponding to the numerical experiment with ${\tt ttol^+}= (0.125)^{11}$). Then for every two consecutive times $t^{k-1},t^k$ we assume that there exists a constant $C_k$ such that
$$\|U_h(t)\|_{L^\infty(\Omega)}=C_k\frac1{(t^*-t)^{p_k}},\quad t=t^{k-1}, t^k,$$
and hence we compute $p_k$ by
$$p_k=\frac{\log(\|U_h(t^k)\|_{L^\infty(\Omega)}/\|U_h(t_{k-1})\|_{L^\infty(\Omega)})}{\log\left((t^*-t^{k-1})/(t^*-t^k)\right)}.$$
This produces a sequence $\{p_k\}_{k=1}^{N}$ of numerical blow-up rates. Since for this example the theoretical blow-up rate is one, for a ``correct'' asymptotic blow-up rate of the numerical approximation we expect $p_k$ to tend to a number close to one as $k\to N$. This is indeed the case, as observed in Figure~\ref{numericalblowuprate}.

\subsection{Example 2}
Let $\Omega = (-4,4)^2$, $\varepsilon = 1$, ${\bf a} = (1,1)^T$, $f_0 = -1$ and $u_0 = 0$. This numerical example is interesting to study as not much is known about blow-up problems with non-symmetric spatial operators. The solution behaves as the solution to a linear convection-diffusion problem for small $t$. As time progresses, the nonlinear term takes over and the solution begins to exhibit point growth leading to blow-up. As in Example 1, we choose to use a small spatial threshold to render the spatial contribution to both the error and the estimator negligible. We then reduce the temporal threshold and observe how far we can advance towards the blow-up time. The results are given in Table \ref{blowupdata2}.
\begin{table}[h!]
\caption{Example 2 Results} % title of Table
\centering % used for centering table
\begin{tabular}{c c c c c} % centered columns (4 columns)
\hline\hline %inserts double horizontal lines
${\tt ttol^+}$ & Time Steps & Estimator & Final Time & $||U_h(T)||_{L^{\infty}(\Omega)}$ \\ % inserts table
%heading
\hline % inserts single horizontal line
1 & 4 & 3.6 & 0.78125 & 0.886 \\ % inserting body of the table
0.125 & 10 & 3.6 & 0.97656 & 1.322 \\
$(0.125)^2$ & 54 & 22.0 & 1.31836 & 3.269 \\
$(0.125)^3$ & 119 & 47.5 & 1.41602 & 5.107 \\
$(0.125)^4$ & 252 & 132.1 & 1.48163 & 8.059 \\
$(0.125)^5$ & 520 & 218.4 & 1.51711 & 11.819 \\
$(0.125)^6$ & 1064 & 664.6 & 1.54467 & 18.139 \\
$(0.125)^7$ & 2158 & 1466.1 & 1.56224 & 27.405 \\
$(0.125)^8$ & 4354 & 1421.7 & 1.57402 & 41.374 \\
$(0.125)^9$ & 8792 & 11423.0 & 1.58243 & 64.450 \\
$(0.125)^{10}$ & 17713 & 21497.8 & 1.58770 & 99.190 \\
$(0.125)^{11}$ & 35580 & 21097.1 & 1.59092 & 145.785 \\
$(0.125)^{12}$ & 71352 & 35862.0 & 1.59299 & 211.278 \\
\hline %inserts single line
\end{tabular}
\label{blowupdata2}
\end{table}
From the data, we conclude that
\begin{equation}
\notag
\|U_h\|_{L^{\infty}(0,T;L^{\infty}(\Omega))} \propto N^{1/2}.
\end{equation}
Although not much is known about blow-up problems with convection, it is reasonable to assume that because the nonlinear term dominates close to the blow-up time that an analogous relationship between the magnitude of the exact solution in the $L^{\infty}(L^{\infty})$-norm and distance from the blow-up time exists as in Example 1. Assuming that this is indeed the case and following the same reasoning as in Example 1, we again conclude that
$$
\lambda({\tt ttol^+},N) \propto N^{-1/2}.
$$

\subsection{Example 3}

Let $\Omega = (-8,8)^2$, $\varepsilon = 1$, ${\bf a} = (0,0)^T$, $f_0 = 0$ and the `volcano' type initial condition be given by $u_0 = 10(x^2+y^2)\mathrm{e}^{-0.5(x^2+y^2)}$. The blow-up set for this example is a circle centred on the origin {\bf --} this induces layer type phenomena in the solution around the blow-up set as the blow-up time is approached making this example a good test of the spatial capabilities of the adaptive algorithm. Once more, we choose a small spatial threshold so that the spatial contribution to the error and the estimator are negligible. We then reduce the temporal threshold and see how far we can advance towards the blow-up time. The results are given in Table \ref{blowupdata3}.

\begin{table}[ht]
\caption{Example 3 Results} % title of Table
\centering % used for centering table
\begin{tabular}{c c c c c} % centered columns (4 columns)
\hline\hline %inserts double horizontal lines
${\tt ttol^+}$ & Time Steps & Estimator & Final Time & $||U_h(T)||_{L^{\infty}(\Omega)}$ \\ % inserts table
%heading
\hline % inserts single horizontal line
8 & 3 & 15 & 0.06250 & 10.371 \\ % inserting body of the table
1 & 10 & 63 & 0.09375 & 14.194 \\
$0.125$ & 36 & 211 & 0.11979 & 21.842 \\
$(0.125)^2$ & 86 & 533 & 0.13412 & 31.446 \\
$(0.125)^3$ & 190 & 971 & 0.14388 & 45.122 \\
$(0.125)^4$ & 404 & 1358 & 0.15072 & 64.907 \\
$(0.125)^5$ & 880 & 5853 & 0.15601 & 98.048 \\
$(0.125)^6$ & 1853 & 10654 & 0.15942 & 146.162 \\
$(0.125)^7$ & 3831 & 21301 & 0.16176 & 219.423 \\
$(0.125)^8$ & 7851 & 143989 & 0.16336 & 332.849 \\
$(0.125)^9$ & 16137 & 287420 & 0.16442 & 505.236 \\
$(0.125)^{10}$ & 32846 & 331848 & 0.16512 & 769.652 \\
$(0.125)^{11}$ & 66442 & 626522 & 0.16558 & 1175.21 \\
\hline %inserts single line
\end{tabular}
\label{blowupdata3}
\end{table}

Once again, the data implies that
\begin{equation}
\notag
\|U_h\|_{L^{\infty}(0,T;L^{\infty}(\Omega))} \propto N^{1/2}.
\end{equation}
Arguing as in Example 1, we again conclude that
\begin{equation}
\begin{aligned}
\notag
\lambda({\tt ttol^+},N) \propto N^{-1/2}.
\end{aligned}
\end{equation}

The numerical solution  at $t=0$ and $t=T$ obtained with the final numerical experiment (${\tt ttol^+} = (0.125)^{11}$) is shown in Figure \ref{blowupprofiles}; the corresponding  meshes are displayed in Figure \ref{blowupmeshes}. The initial mesh has a relatively homogenous distribution of elements which is to be expected since the initial condition is relatively smooth. In the final mesh, elements have been added in the vicinity of the blow-up set and removed elsewhere, notably near the origin. The distribution of elements in the final mesh strongly indicates that the adaptive algorithm is adding and removing elements in an efficient manner.

\begin{figure}
\centering
\includegraphics[scale=0.48]{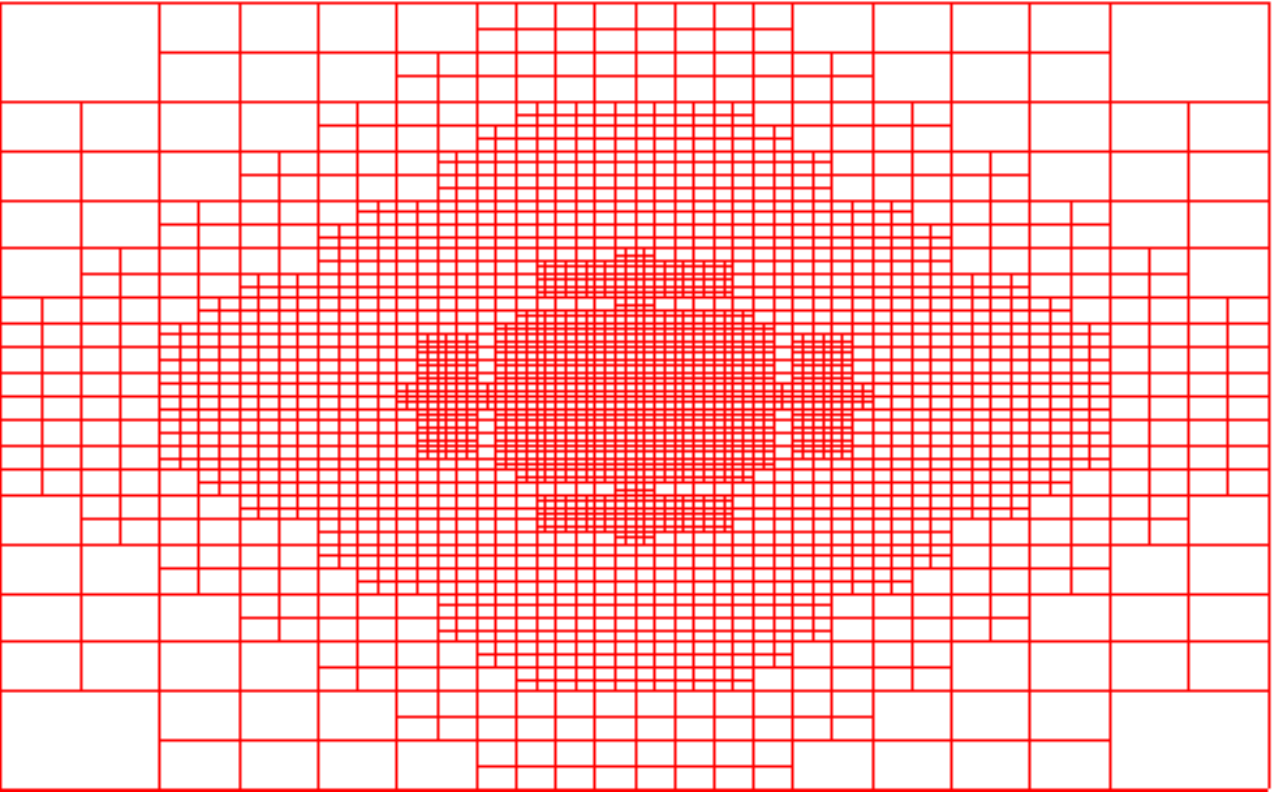} \, \includegraphics[scale=0.48]{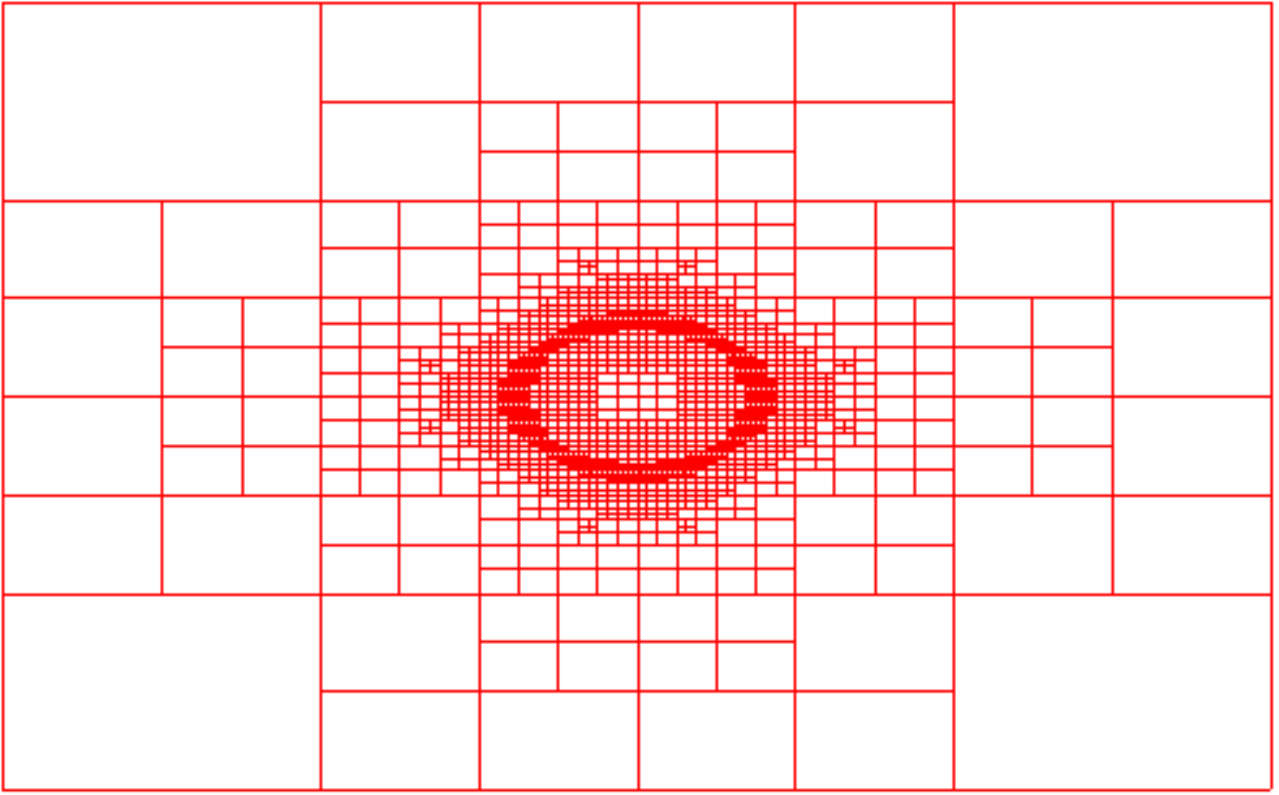}
\caption{Example 3: Initial (left) and final (right) meshes.}
\label{blowupmeshes}
\end{figure}
\begin{figure}
\centering
\includegraphics[scale=0.23]{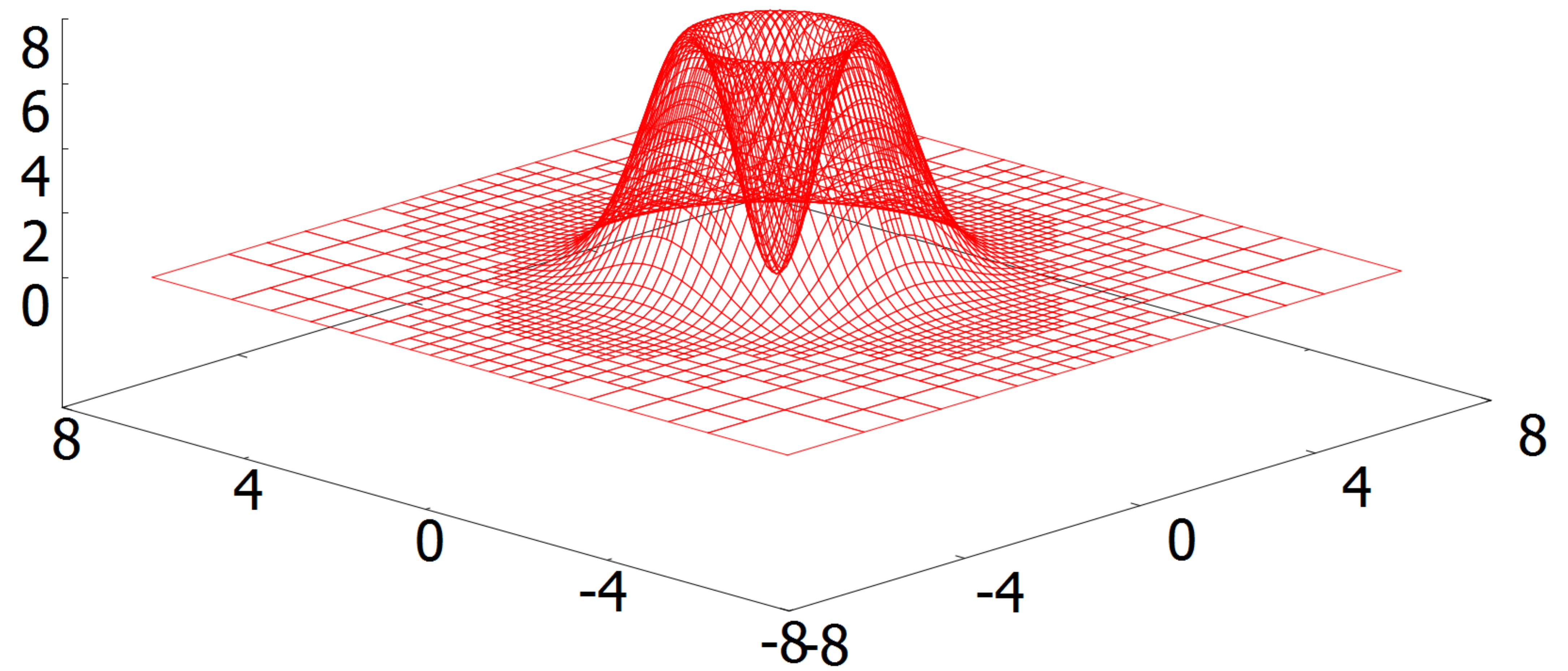} \vspace{3mm} 
\includegraphics[scale=0.23]{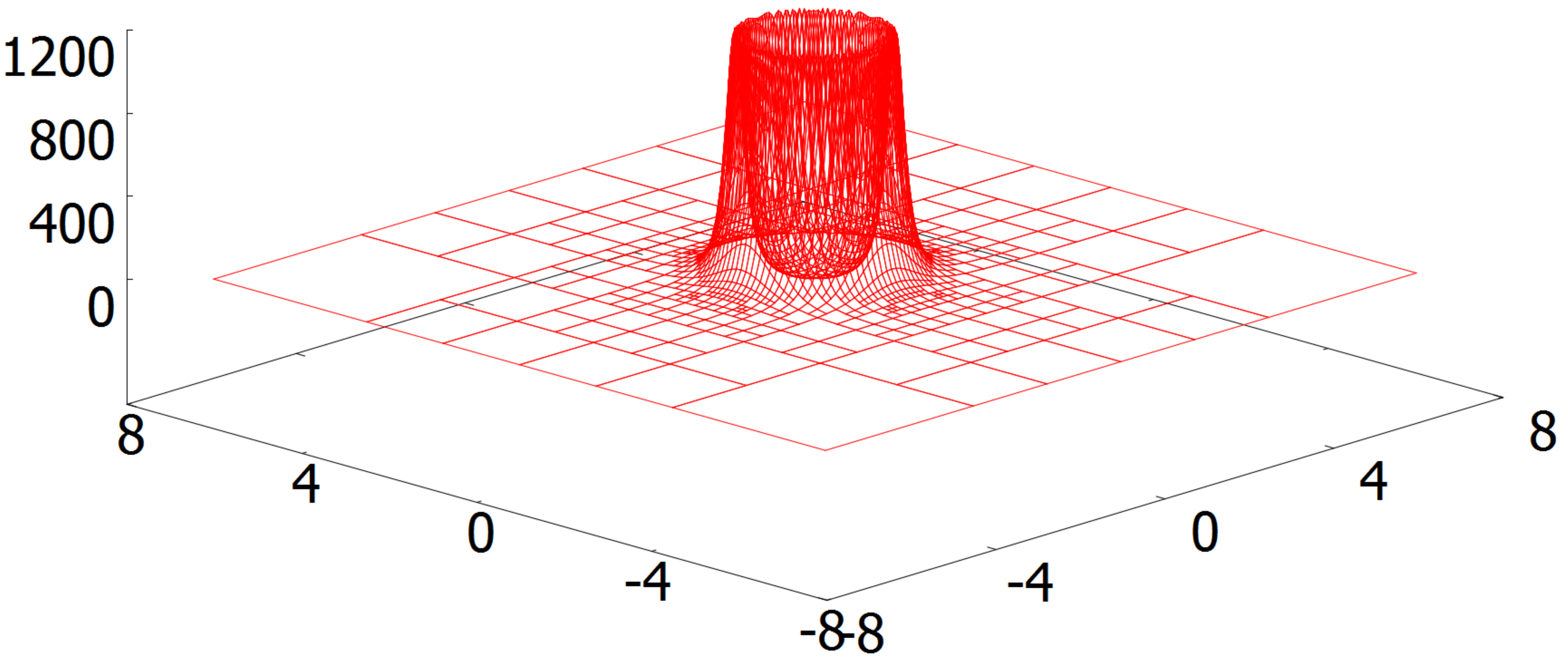}
\caption{Example 3: Initial (above) and final (below) solution profiles.}
\label{blowupprofiles}
\end{figure}

\section{Conclusions}\label{conclusions_sec}

We proposed a framework for space-time adaptivity based on rigorous a posteriori bounds for an IMEX dG discretization of a semilinear blow-up problem. The error estimator was applied to a number of test problems and appears to converge towards the blow-up time in all cases. In Section \ref{ODE_sec}, it was observed that the a posteriori error estimator for the related ODE problem with polynomial nonlinearity approaches the blow-up time with a rate of at least one for a basic Euler method. The numerical examples show that, for the PDE blow-up problem, the proposed error estimator appears to be advancing towards the blow-up time at a rate approximately half of that observed for the corresponding ODE error estimator. A possible reason for this behaviour lies in the proof of the a posteriori bound via an energy argument. Nevertheless, it is this very energy argument which delivers a \emph{practical} conditional a posteriori bound in the sense that condition \eqref{cond2} can be satisfied within a practically relevant (in terms of computational cost) mesh-parameter regime. It would be interesting to investigate the derivation of conditional a posteriori bounds for fully-discrete schemes for blow-up problems via semigroup techniques, in the spirit of \cite{K01}, although this currently remains a challenging task.

\section*{Acknowledgements}

Irene Kyza was supported in part  by the  European Social Fund (ESF) {\bf --} European Union (EU) and National Resources of the Greek State within the framework of the Action ``Supporting Postdoctoral Researchers''  of the Operational Programme ``Education and Lifelong Learning (EdLL)''. Stephen Metcalfe gratefully acknowledges the funding of the Engineering and Physical Sciences Research Council (EPSRC).  This work originated from a number of visits of the authors to the Archimedes Center for Modelling, Analysis \& Computation (ACMAC), which we gratefully acknowledge. We would also like to thank Prof. Theodoros Katsaounis of the University of Crete for suggesting the final numerical example.

\bibliographystyle{siam}
\bibliography{bibliography}

\begin{thebibliography}{10}

\bibitem{ADR}
{\sc G.~Acosta, R.~G. Duran, and J.~D. Rossi}, {\em An adaptive time step
  procedure for a parabolic problem with blow-up}, Computing, 68 (2002),
  pp.~343--373.

\bibitem{ADKM03}
{\sc Georgias~D. Akrivis, V.A. Dougalis, Ohannes~A. Karakashian, and W.R.
  McKinney}, {\em Numerical approximation of blow-up of radially symmetric
  solutions of the nonlinear {S}chr{\"o}dinger equation}, SIAM J. Sci. Comput.,
  25 (2003), pp.~186--212.

\bibitem{ABN09}
{\sc Louis~A. Assal{\'e}, Th{\'e}odore~K. Boni, and Diabate Nabongo}, {\em
  Numerical blow-up time for a semilinear parabolic equation with nonlinear
  boundary conditions}, J. Appl. Math., 2008 (2009).

\bibitem{BHK07}
{\sc W.~Bangerth, R.~Hartmann, and G.~Kanschat}, {\em deal.{II}---a
  general-purpose object-oriented finite element library}, ACM Trans. Math.
  Software, 33 (2007), pp.~Art. 24, 27.

\bibitem{B05}
{\sc S{\"o}ren Bartels}, {\em A posteriori error analysis for time-dependent
  {G}inzburg-{L}andau type equations}, Numer. Math., 99 (2005), pp.~557--583.

\bibitem{BM11}
{\sc S{\"o}ren Bartels and R{\"u}diger M{\"u}ller}, {\em Quasi-optimal and
  robust a posteriori error estimates in ${L}^{\infty}({L}^2)$ for the
  approximation of {A}llen-{C}ahn equations past singularities}, Math. Comp.,
  80 (2011), pp.~761--780.

\bibitem{BBM05}
{\sc Amal Bergam, Christine Bernardi, and Zoubida Mghazli}, {\em A posteriori
  analysis of the finite element discretization of some parabolic equations},
  Math. Comp., 74 (2005), pp.~1117--1138.

\bibitem{BK88}
{\sc Marsha Berger and Robert~V. Kohn}, {\em A rescaling algorithm for the
  numerical calculation of blowing-up solutions}, Comm. Pure Appl. Math., 41
  (1988), pp.~841--863.

\bibitem{BHR96}
{\sc Chris~J. Budd, Weizhang Huang, and Robert~D. Russell}, {\em Moving mesh
  methods for problems with blow-up}, SIAM J. Sci. Comput., 17 (1996),
  pp.~305--327.

\bibitem{CGM14}
{\sc Andrea Cangiani, Emmanuil~H. Georgoulis, and Stephen Metcalfe}, {\em
  Adaptive discontinuous {G}alerkin methods for nonstationary
  convection-diffusion problems}, IMA J. Numer. Anal., 34 (2014),
  pp.~1578--1597.

\bibitem{CEGST14}
{\sc J.H. Chaudry, D.~Estep, V.~Ginting, J.N. Shadid, and S.~Tavener}, {\em A
  posteriori error analysis of {IMEX} multi-step time integration methods for
  advection-diffusion-reaction equations}, Submitted for publication,  (2014).

\bibitem{CS02}
{\sc James Coleman and Catherine Sulem}, {\em Numerical simulation of blow-up
  solutions of the vector nonlinear {S}chr{\"o}dinger equation}, Phys. Rev. E,
  66 (2002), p.~036701.

\bibitem{CM}
{\sc E.~Cuesta and C.~Makridakis}, {\em A posteriori error estimates and
  maximal regularity for approximations of fully nonlinear parabolic problems
  in banach spaces}, Numer. Math., 110 (2008), pp.~257--275.

\bibitem{DGN07}
{\sc Javier De~Frutos, Bosco Garc{\'\i}a-Archilla, and Julia Novo}, {\em A
  posteriori error estimates for fully discrete nonlinear parabolic problems},
  Comput. Methods Appl. Mech. Engrg., 196 (2007), pp.~3462--3474.

\bibitem{DPMF05}
{\sc Arturo de~Pablo, Mayte P{\'e}rez-Llanos, and Ra{\'u}l Ferreira}, {\em
  Numerical blow-up for the p-{L}aplacian equation with a nonlinear source}, in
  Proceedings of the 11th International Conference on Differential Equations
  (Equadiff’05), 2005, pp.~363--367.

\bibitem{DG01}
{\sc A.~Demlow and E.~H. Georgoulis}, {\em Pointwise a posteriori error control
  for discontinuous {G}alerkin methods for elliptic problems}, SIAM J. Numer.
  Anal., 50 (2012), pp.~2159--2181.

\bibitem{DE12}
{\sc Daniele~Antonio Di~Pietro and Alexandre Ern}, {\em Mathematical aspects of
  discontinuous {G}alerkin methods}, vol.~69 of Math\'ematiques \& Applications
  (Berlin) [Mathematics \& Applications], Springer, Heidelberg, 2012.

\bibitem{DKKV98}
{\sc Stefka Dimova, Michael Kaschiev, Milena Koleva, and Daniela Vasileva},
  {\em Numerical analysis of radially nonsymmetric blow-up solutions of a
  nonlinear parabolic problem}, J. Comput. Appl. Math., 97 (1998), pp.~81--97.

\bibitem{D01}
{\sc Sever~Silvestru Dragomir}, {\em Some Gronwall type inequalities and
  applications}, Nova Science Publishers, 2003.

\bibitem{FGR02}
{\sc Raul Ferreira, Pablo Groisman, and Julio~D. Rossi}, {\em Numerical blow-up
  for a nonlinear problem with a nonlinear boundary condition}, Math. Models
  Methods Appl. Sci., 12 (2002), pp.~461--483.

\bibitem{FI03}
{\sc Gadi Fibich and Boaz Ilan}, {\em Discretization effects in the nonlinear
  {S}chr{\"o}dinger equation}, Appl. Numer. Math., 44 (2003), pp.~63--75.

\bibitem{Fierro}
{\sc C.~Fierro and A.~Veeser}, {\em On the a posteriori error analysis of
  equations of prescribed mean curvature}, Math. Comp., 72 (2003),
  pp.~1611--1634.

\bibitem{GLV11}
{\sc Emmanuil~H. Georgoulis, Omar Lakkis, and Juha~M. Virtanen}, {\em A
  posteriori error control for discontinuous {G}alerkin methods for parabolic
  problems}, SIAM J. Numer. Anal., 49 (2011), pp.~427--458.

\bibitem{GM14}
{\sc Emmanuil~H. Georgoulis and Charalambos Makridakis}, {\em On a posteriori
  error control for the {A}llen-{C}ahn problem}, Math. Method. Appl. Sci., 37
  (2014), pp.~173--179.

\bibitem{G1}
{\sc P.~Groisman}, {\em Totally discrete explicit and semi-implicit euler
  methods for a blow-up problem in several space dimensions}, Computing, 76
  (2006), pp.~325--352.

\bibitem{H06}
{\sc Chiaki Hirota and Kazufumi Ozawa}, {\em Numerical method of estimating the
  blow-up time and rate of the solution of ordinary differential equations {\bf
  --} an application to the blow-up problems of partial differential
  equations}, J. Comput. Appl. Math., 193 (2006), pp.~614 -- 637.

\bibitem{Hu}
{\sc Bei Hu}, {\em Blow-up theories for semilinear parabolic equations.},
  vol.~2018 of Lecture Notes in Mathematics, Springer, Heidelberg, 2011.

\bibitem{HMR08}
{\sc Weizhang Huang, Jingtang Ma, and Robert~D. Russell}, {\em A study of
  moving mesh {PDE} methods for numerical simulation of blowup in reaction
  diffusion equations}, J. Comput. Phys., 227 (2008), pp.~6532--6552.

\bibitem{JW14}
{\sc B.~{Janssen} and T.~P. {Wihler}}, {\em {Existence Results for the
  Continuous and Discontinuous {G}alerkin Time Stepping Methods for Nonlinear
  Initial Value Problems}}, ArXiv e-prints,  (2014).

\bibitem{KP03}
{\sc Ohannes~A. Karakashian and Frederic Pascal}, {\em A posteriori error
  estimates for a discontinuous {G}alerkin approximation of second-order
  elliptic problems}, SIAM J. Numer. Anal., 41 (2003), pp.~2374--2399
  (electronic).

\bibitem{KNS04}
{\sc Daniel Kessler, Ricardo~H. Nochetto, and Alfred Schmidt}, {\em A
  posteriori error control for the {A}llen-{C}ahn problem: circumventing
  {G}ronwall'��s inequality}, ESAIM Math. Model. Numer. Anal., 38 (2004),
  pp.~129--142.

\bibitem{KMR11}
{\sc Christian Klein, Benson Muite, and Kristelle Roidot}, {\em Numerical study
  of blowup in the {D}avey-{S}tewartson system}, arXiv preprint
  arXiv:1112.4043,  (2011).

\bibitem{K01}
{\sc I.~Kyza and C.~Makridakis}, {\em Analysis for time discrete approximations
  of blow-up solutions of semilinear parabolic equations}, SIAM J. Numer.
  Anal., 49 (2011), pp.~405--426.

\bibitem{LM06}
{\sc Omar Lakkis and Charalambos Makridakis}, {\em Elliptic reconstruction and
  a posteriori error estimates for fully discrete linear parabolic problems},
  Math. Comp., 75 (2006), pp.~1627--1658.

\bibitem{LaNo}
{\sc O.~Lakkis and R.~H. Nochetto}, {\em A posteriori error analysis for the
  mean curvature flow of graphs}, SIAM J. Numer. Anal., 42 (2005),
  pp.~1875--1898.

\bibitem{MN03}
{\sc Charalambos Makridakis and Ricardo~H. Nochetto}, {\em Elliptic
  reconstruction and a posteriori error estimates for parabolic problems}, SIAM
  J. Numer. Anal., 41 (2003), pp.~1585--1594.

\bibitem{MN2}
{\sc C.~Makridakis and R.~H. Nochetto}, {\em A posteriori error analysis for
  higher order dissipative methods for evolution problems}, Numer. Math., 104
  (2006), pp.~489--514.

\bibitem{MZ98}
{\sc Frank Merle and Hatem Zaag}, {\em Optimal estimates for blowup rate and
  behavior for nonlinear heat equations}, Communications on pure and applied
  mathematics, 51 (1998), pp.~139--196.

\bibitem{MerleZaag}
\leavevmode\vrule height 2pt depth -1.6pt width 23pt, {\em A {L}iouville
  theorem for vector-valued nonlinear heat equations and applications}, Math.
  Ann., 316 (2000), pp.~103--137.

\bibitem{NB11}
{\sc F.~K. N'Gohisse and Th{\'e}odore~K. Boni}, {\em Numerical blow-up for a
  nonlinear heat equation}, Acta Math. Sin. (Engl. Ser.), 27 (2011),
  pp.~845--862.

\bibitem{NZ14}
{\sc V.T. Nguyen and H.~Zaag}, {\em Blow-up results for a strongly perturbed
  semilinear heat equation: Theoretical analysis and numerical method},
  (2014).

\bibitem{Plex}
{\sc Michael Plexousakis}, {\em An adaptive nonconforming finite element method
  for the nonlinear {S}chr\"{o}dinger equation}, PhD thesis, University of
  {T}ennessee, 1996.

\bibitem{SZ09}
{\sc Dominik Sch{\"o}tzau and Liang Zhu}, {\em A robust a-posteriori error
  estimator for discontinuous {G}alerkin methods for convection-diffusion
  equations}, Appl. Numer. Math., 59 (2009), pp.~2236--2255.

\bibitem{SF90}
{\sc A.M. Stuart and M.S. Floater}, {\em On the computation of blow-up}, Euro.
  J. Appl. Math., 1 (1990), pp.~47--71.

\bibitem{TS92}
{\sc Y.~Tourigny and J.M. Sanz-Serna}, {\em The numerical study of blowup with
  application to a nonlinear {S}chr{\"o}dinger equation}, J. Comput. Phys., 102
  (1992), pp.~407--416.

\bibitem{V98f}
{\sc R.~Verf{\"u}rth}, {\em A posteriori error estimates for nonlinear
  problems: ${L}^r(0, t; {L}^p({\Omega}))$-error estimates for finite element
  discretizations of parabolic equations}, Math. Comp., 67 (1998),
  pp.~1335--1360.

\bibitem{V98e}
\leavevmode\vrule height 2pt depth -1.6pt width 23pt, {\em A posteriori error
  estimates for nonlinear problems: ${L}^r(0, t; {W}^{1,p}({\Omega}))$-error
  estimates for finite element discretizations of parabolic equations}, Numer.
  Methods Partial Differential Equations, 14 (1998), pp.~487--518.

\bibitem{V04}
\leavevmode\vrule height 2pt depth -1.6pt width 23pt, {\em A posteriori error
  estimates for non-linear parabolic equations}, Preprint, Ruhr-Universit{\"a}t
  Bochum, Fakult{\"a}t f{\"u}r Mathematik, Bochum, Germany,  (2004).

\end{thebibliography}

\end{document}